\documentclass[a4paper,11pt,reqno]{amsart}
\usepackage[centering, totalwidth = 390pt, totalheight = 610pt]{geometry}
\usepackage{amsmath,amssymb,amsfonts, url,stmaryrd}
\usepackage{eucal}
\usepackage{verbatim}
\usepackage{amsthm}
\usepackage{xspace}

\usepackage{enumerate}

\numberwithin{equation}{section}

\theoremstyle{plain}

\newtheorem{Theorem}{Theorem}

\newtheorem{Corollary}[Theorem]{Corollary}
\newtheorem{Proposition}[Theorem]{Proposition}
\newtheorem{Lemma}[Theorem]{Lemma}

\theoremstyle{definition}

\newcommand{\abs}[1]{{\left|{#1}\right|}}

\newcommand{\Cat}{\ensuremath{\textnormal{Cat}}\xspace}
\newcommand{\Icon}{\ensuremath{\textnormal{Icon}}\xspace}
\newcommand{\Ps}{\ensuremath{\textnormal{Ps}}\xspace}

\newcommand{\Set}{\ensuremath{\textnormal{Set}}\xspace}
\newcommand{\f}[1]{\ensuremath{\mathcal{#1}}\xspace}

\newcommand{\TAlg}{\ensuremath{\textnormal{T-Alg}}\xspace}
\newcommand{\TAlgs}{\ensuremath{\textnormal{T-Alg}_{\textnormal{s}}}\xspace}

\newcommand{\thg}{{\mathord{\text{--}}}}
\newcommand{\qun}{q}
\newcommand{\ToAlg}{\ensuremath{\text{T}_0\textnormal{-Alg}}\xspace}
\newcommand{\TpAlg}{\ensuremath{\text{T}_d\textnormal{-Alg}}\xspace}
\newcommand{\ZKoAlg}{\ensuremath{\text{ZK}_0\textnormal{-Alg}}\xspace}
\newcommand{\ZKpAlg}{\ensuremath{\text{ZK}_d\textnormal{-Alg}}\xspace}
\newcommand{\CatGph}{\ensuremath{\textnormal{Cat-Gph}}\xspace}
\newcommand{\End}{\ensuremath{\textnormal{End}}\xspace}
\newcommand{\Mnd}{\ensuremath{\textnormal{Mnd}}\xspace}

\newcommand{\Mndf}{\ensuremath{\textnormal{Mnd}_\mathrm{f}}\xspace}
\newcommand{\Endsf}{\ensuremath{\textnormal{End}_\mathrm{sf}}\xspace}
\newcommand{\Mndsf}{\ensuremath{\textnormal{Mnd}_\mathrm{sf}}\xspace}

\input xy
\xyoption{all}
\SelectTips{cm}{11}
\makeatletter
\def\matrixobject@{%
 \edef \next@{={\DirectionfromtheDirection@ }}%
 \expandafter \toks@ \next@ \plainxy@
 \let\xy@@ix@=\xyq@@toksix@
 \xyFN@ \OBJECT@}
\let\xy@entry@@norm=\entry@@norm
\def\entry@@norm@patched{%
 \let\object@=\matrixobject@
 \xy@entry@@norm }
\AtBeginDocument{\let\entry@@norm\entry@@norm@patched}
\makeatother
\newcommand{\twocong}[2][0.5]{\ar@{}[#2] \save ?(#1)*{\cong}\restore}
\newcommand{\twoeq}[2][0.5]{\ar@{}[#2] \save ?(#1)*{=}\restore}
\newcommand{\rtwocell}[3][0.5]{\ar@{}[#2] \ar@{=>}?(#1)+/l 0.2cm/;?(#1)+/r 0.2cm/^{#3}}
\newcommand{\ltwocell}[3][0.5]{\ar@{}[#2] \ar@{=>}?(#1)+/r 0.2cm/;?(#1)+/l 0.2cm/^{#3}}
\newcommand{\ltwocello}[3][0.5]{\ar@{}[#2] \ar@{=>}?(#1)+/r 0.2cm/;?(#1)+/l 0.2cm/_{#3}}
\newcommand{\dtwocell}[3][0.5]{\ar@{}[#2] \ar@{=>}?(#1)+/u  0.2cm/;?(#1)+/d 0.2cm/^{#3}}
\newcommand{\dltwocell}[3][0.5]{\ar@{}[#2] \ar@{=>}?(#1)+/ur  0.2cm/;?(#1)+/dl 0.2cm/^{#3}}
\newcommand{\drtwocell}[3][0.5]{\ar@{}[#2] \ar@{=>}?(#1)+/ul  0.2cm/;?(#1)+/dr 0.2cm/^{#3}}
\newcommand{\dthreecell}[3][0.5]{\ar@{}[#2] \ar@3{->}?(#1)+/u  0.2cm/;?(#1)+/d 0.2cm/^{#3}}
\newcommand{\utwocell}[3][0.5]{\ar@{}[#2] \ar@{=>}?(#1)+/d 0.2cm/;?(#1)+/u 0.2cm/_{#3}}
\newcommand{\dtwocelltarg}[3][0.5]{\ar@{}#2 \ar@{=>}?(#1)+/u  0.2cm/;?(#1)+/d 0.2cm/^{#3}}
\newcommand{\utwocelltarg}[3][0.5]{\ar@{}#2 \ar@{=>}?(#1)+/d  0.2cm/;?(#1)+/u 0.2cm/_{#3}}
\newcommand{\cd}[2][]{\vcenter{\hbox{\xymatrix#1{#2}}}}
\newcommand{\op}{\mathrm{op}}
\newcommand{\us}{U}
\newcommand{\fs}{F}
\newcommand{\uu}{U}
\newcommand{\uo}{U_0}
\newcommand{\fo}{F_0}
\newcommand{\up}{U_d}
\newcommand{\fp}{F_d}
\newcommand{\To}{T_0}
\newcommand{\Tp}{T_d}
\newcommand{\Co}{{\f C}_0}
\newcommand{\Cp}{{\f C}_d}
\newcommand{\proj}{discrete}
\newcommand{\projs}{discretes}

\title{On semiflexible, flexible and pie algebras}
\begin{document}
\author{John Bourke}
\address{Department of Mathematics and Statistics, Masaryk University, Kotl\'a\v rsk\'a 2, Brno 60000, Czech Republic}
\email{bourkej@math.muni.cz} 
\author{Richard Garner}
\address{Department of Computing, Macquarie University, NSW 2109, Australia}
\email{richard.garner@mq.edu.au}

\date{\today}

\thanks{The first author acknowledges the support of the Eduard \v Cech Center for Algebra and
Geometry, grant number LC505.
The second author acknowledges the support of an Australian Research Council Discovery Project, grant number DP110102360.}

\begin{abstract}
We introduce the notion of \emph{pie algebra} for a $2$-monad, these bearing the same relationship to the flexible and semiflexible algebras as pie limits do to flexible and semiflexible ones. We see that in many cases, the pie algebras are precisely those ``free at the level of objects'' in a suitable sense; so that, for instance, a strict monoidal category is pie just when its underlying monoid of objects is free. Pie algebras are contrasted with flexible and semiflexible algebras via a series of characterisations of each class; particular attention is paid to the case of pie, flexible and semiflexible weights, these being characterised in terms of the behaviour of the corresponding weighted limit functors.
\end{abstract}
 \leftmargini=2em

\def\xypic{\hbox{\rm\Xy-pic}}
\maketitle
\section{Introduction}
One category-theoretic approach to universal algebra is based on the theory of monads. Single-sorted (possibly infinitary) algebraic theories correspond with monads on the category of sets, and models of a theory with algebras for the associated monad; this justifies our regarding monads on other categories as generalised algebraic theories, and many basic aspects of classical universal algebra may be reconstructed in this broader context. A further generalisation is obtained on passing from the study of monads on categories to that of $2$-monads on $2$-categories, which yields a kind of ``two-dimensional universal algebra''. The simplest case studies $2$-monads on $\Cat$, which encode many familiar structures that may be borne by a category: thus there are $2$-monads whose algebras are monoidal categories, or categories with finite products, or distributive categories, or cocomplete categories, and so on.

In the passage from $1$- to $2$-dimensional monad theory, a number of new phenomena come into being. As well as \emph{strict} algebras for a $2$-monad, we also have \emph{pseudo} and \emph{lax} ones, for which the algebra axioms have been weakened to hold only up to coherent $2$-cells, invertible in the former case but not in the latter. Likewise, between the algebras for a $2$-monad we have not only the \emph{strict} morphisms, but also \emph{pseudo} and \emph{lax} ones, which preserve the algebra structure in correspondingly weakened manners. 
The interplay between strict, pseudo and lax in $2$-dimensional monad theory provides an abstract setting for the study of coherence problems of the kind exemplified by Mac Lane's famous result \cite{Mac-Lane1963Natural} that \emph{every monoidal category is monoidally equivalent to a strict one}. The study of coherence from the standpoint of $2$-monad theory was championed by Max Kelly, who initiated the programme in~\cite{Kelly1974Coherence, Kelly1974Doctrinal} and with his collaborators, brought it to a particularly fine expression in~\cite{Blackwell1989Two-dimensional}. 

A subtle and crucial  notion in Kelly's framework is that of \emph{flexibility}; this first arose in~\cite{Kelly1974Doctrinal}, where was introduced the notion of a \emph{flexible $2$-monad}. One of the most important properties of flexible $2$-monads, as described in~\cite[Theorem 3.3]{Kelly2004Monoidal}, is that strict algebra structure may be transported along equivalences: which is to say that if $A \simeq B$, and $A$ bears strict algebra structure for a flexible $2$-monad, then so does $B$. This is the case, for example, with the $2$-monad on $\Cat$ whose algebras are monoidal categories, but not for the $2$-monad whose algebras are \emph{strict} monoidal categories; and indeed, the former $2$-monad is flexible, whilst the latter is not. 
This particular good behaviour of flexible $2$-monads is a consequence of a more fundamental one: that \emph{every pseudoalgebra for a flexible $2$-monad is isomorphic to a strict one}. Intuitively, we think that flexibility of a $2$-monad expresses a certain ``looseness'' in the structure it imposes on its algebras, and this intuition has been expressed for $2$-monads on $\Cat$ in the following way: that such a $2$-monad is flexible if its algebras can be presented as categories equipped with basic operations and with natural transformations between derived operations, satisfying equations between derived natural transformations but with no equations being imposed between derived operations themselves.

The scope of the concept of flexibility was vastly expanded in~\cite{Blackwell1989Two-dimensional}, where was introduced the notion of a \emph{flexible algebra} for a $2$-monad with rank on a complete and cocomplete $2$-category. With suitable cardinality constraints, $2$-monads may themselves be viewed as algebras for such a $2$-monad, so that flexible $2$-monads are an instance of the more general notion. The flexible algebras exhibit the same kind of ``looseness'' as we saw above, this now being manifested in the result that \emph{each pseudomorphism out of a flexible algebra is isomorphic to a strict one}; which is in fact the same property that allows pseudoalgebras to be replaced by isomorphic strict ones for a flexible $2$-monad. We have the intuitive picture that an algebra for a $2$-monad on $\Cat$ is flexible if it admits a presentation in which no equalities are forced at the level of objects: so that, for example, a monoidal category is flexible if it may be obtained from a free one through the addition of new morphisms and new equations between morphisms, but without any new equations being added between objects.

The notion of flexibility is natural from a $2$-categorical perspective, but also from a homotopy-theoretic one; Lack showed in~\cite{Lack2007Homotopy-theoretic} that, for a $2$-monad $T$ with rank on a locally presentable $2$-category $\f C$, the $2$-category of strict $T$-algebras and strict algebra morphisms bears a Quillen model structure whose cofibrant objects are precisely the flexible algebras. Many results concerning flexibility can be understood from this homotopy-theoretic perspective, describing as they do certain good properties of cofibrant objects. 

We have above discussed the notion of flexibility in pragmatic terms, showing that it is strong enough to imply certain desirable properties, and broad enough to admit all the examples falling under our intuitive picture of flexibility. Yet in doing so we have glossed over a gap in our understanding. On the one hand, it was shown in~\cite[Theorem 4.7]{Blackwell1989Two-dimensional} that the ability to replace pseudomorphisms by strict ones in the manner described above characterises not the flexible algebras, but the larger class of \emph{semiflexible algebras}; these being the closure of the flexible algebras amongst all algebras under equivalences. On the other, the algebras answering to our intuitive picture of flexibility, as being those which are ``free at the level of objects'', need not comprise the totality of the flexible algebras, since the flexibles are closed under retracts, whilst those those which are free in the manner just mentioned may not be.
We thus have three classes of algebras, each containing the next: the semiflexibles, the flexibles, and a third class, as yet unnamed, comprising those algebras ``free at the level of objects''; and whilst there is a substantial body of results concerning the flexibles, we do not have a good grasp of what is lost from these results on moving to the larger class of semiflexibles, or of what is gained on passing to the smaller third class in which  many examples of flexible algebras reside. One of the two main objectives of this paper is to rectify this, by proving a number of theorems which completely characterise the algebras in the three classes. Each of these theorems considers a particular kind of good behaviour that an algebra may have, and gives three increasingly strong forms of that behaviour which are respectively equivalent to the algebra's lying in the first, second or third of our classes. We will also describe closure properties by which we can recognise that an algebra constructed in a particular way must lie in one of these classes, and consider the extent to which these closure properties, in turn, completely characterise the given classes.

Yet before we can do any of these things, we need to give a precise definition for the last of our three classes. Clearly, to speak of an algebra for a $2$-monad's being ``free at the level of objects'' is vague; and even if made precise, it would be insufficiently general, limiting us to $2$-monads on $\Cat$ and similarly well-behaved $2$-categories. The second main objective of this paper, then, is to give a fully general definition of this class, and to analyse its scope for a range of $2$-monads of interest. The algebras in this class will be the \emph{pie algebras} of our title; and we will see that in many cases, they capture precisely our intuitive notion of an algebra ``free at the level of objects''.

The motivation for the name pie algebra comes from the theory of $2$-categorical limits. The relevant limits in this context are the \emph{weighted limits} of~\cite[Chapter~3]{Kelly1982Basic}, in specifying which one gives not only a diagram $D \colon \f J \to \f C$ over which a universal cone is to be constructed, but also a \emph{weight} $W \in [\f J, \Cat]$ specifying the nature of the cones amongst which the universal one is to be sought. For a fixed $\f J$, the $2$-category of weights $[\f J, \Cat]$ can be viewed as the $2$-category of algebras for a $2$-monad on $[\mathrm{ob}\ \f J, \Cat]$, so allowing us to speak of flexible and semiflexible weights---those which are flexible or semiflexible as algebras---and weighted by these, semiflexible and flexible limits. These last were studied in~\cite{Bird1989Flexible}, where it was shown that, amongst other things, the flexible limits are precisely those constructible from products, inserters, equifiers and splittings of idempotents; see~\cite{Kelly1989Elementary} for the definitions of these limit-types.

\emph{Pie limits} are, by definition, those constructible from products, inserters and equifiers alone. Many interesting 2-categorical limits are pie---for example, comma objects, inverters, descent objects, Eilenberg-Moore objects of monads, and pseudo, lax and oplax limits---and experience shows that  pie limits are characterised by the property that \emph{limiting cones force no equations between $1$-cells}. When this property of pie limits is re-expressed as a property of the defining weights, it becomes the statement that the pie weights are those which are ``free at the level of objects'', in a sense which was made precise in~\cite[Corollary 3.3]{Power1991A-characterization}. In other words, when we view weights as algebras for a $2$-monad, those answering to the intuitive description of our third class  are precisely the pie weights; so motivating our naming this third class of algebras, the pie algebras. Of course, for this to be a consistent notation, we must ensure that the pie algebras for the weight $2$-monad are precisely the pie weights. This was in fact done in~\cite{Lack2011Enhanced}; we give an alternative proof in Section~\ref{sec:exs} below.

Let us now give a more detailed account of the content of this paper. We begin in Section~\ref{sec:pie} by defining the notion of pie algebra for a $2$-monad $T$ with rank on a complete and cocomplete $2$-category $\f C$. For such a $T$, it was shown in~\cite{Blackwell1989Two-dimensional} that the inclusion of $\TAlgs$, the $2$-category of strict $T$-algebras and strict algebra morphisms, into $\TAlg$, the $2$-category of strict $T$-algebras and algebra pseudomorphisms, admits a left adjoint $Q \colon \TAlg \to \TAlgs$; and it is in terms of this $Q$ that the authors of ~\cite{Blackwell1989Two-dimensional} defined their notions of flexibility and semiflexibility. We recast these definitions in terms of the $2$-comonad on $\TAlgs$ induced by $Q$ and its right adjoint. We will see that an algebra is semiflexible just when it is a pseudocoalgebra for this $2$-comonad, and flexible just when it is a normalised pseudocoalgebra: and this leads us to define an algebra to be \emph{pie} when it admits strict coalgebra structure. The remainder of Section~\ref{sec:pie} is devoted to further analysis of the notion of pie algebra, leading to our first main theorem, which provides conditions under which our intuitive picture of the pie algebras, as those which are ``free at the level of objects'', is confirmed.

We make use of this theorem in Section~\ref{sec:exs}, where we describe the pie algebras for a range of $2$-monads of interest. In particular, we see that a weight $W \in [\f J, \Cat]$ is a pie algebra for the weight $2$-monad on $[\mathrm{ob}\ \f J, \Cat]$ just when it is a pie weight, so confirming the consistency of our terminology with its motivating case.  Further examples show that a monoidal category is pie just when its monoidal structure is free at the level of objects,  that a $2$-category is pie just when its underlying category is free on a graph, and so on. 

In Section~\ref{sec:char}, we turn to the second of our main objectives, that of clarifying the relationship between the semiflexible, flexible and pie algebras for a $2$-monad. As explained above, we do so through a number of theorems that characterise the algebras in these three classes in terms of certain good behaviours they may possess. In fact, we shall give special consideration to the case of weights, providing alternate forms of our results which characterise the semiflexible, flexible and pie weights in terms of the behaviour of the corresponding weighted limit functor. Let us now give a brief overview of the theorems we will prove.

The first of these builds on the characterisation of the semiflexible algebras described above: an algebra is semiflexible just when every pseudomorphism out of it may be replaced by an isomorphic strict one. The corresponding results for flexible and pie algebras require that this replacement should be done in increasingly well-behaved ways; whilst the corresponding results for semiflexible, flexible and pie weights concern the manner in which weighted pseudocones may be replaced by strict ones. As an application of this result, we show that the semiflexible weights are precisely those whose corresponding limits are also bilimits.

Our second characterisation result takes its most intuitive form for the pie weights: we show that a weight $W \in [\f J, \Cat]$ is pie just when the limit $2$-functor $\{W, \thg\} \colon [\f J, \f K] \to \f K$ admits an extension to a $2$-functor $\Ps(\f J, \f K) \to \f K$ for every complete $2$-category $\f K$; here $\Ps(\f J, \f K)$ denotes the $2$-category of $2$-functors, pseudonatural transformations and modifications. The corresponding results for flexible and semiflexible weights ask for correspondingly weaker extensions of $\{W, \thg\}$ to $\Ps(\f J, \f K)$; whilst those concerning an algebra $A$ for a general $2$-monad deal with extensions of the hom $2$-functor $\TAlgs(A, \thg) \colon \TAlgs \to \Cat$ to $\TAlg$.

Our third result also characterises each class of algebras in terms of the behaviour of the hom $2$-functor $\TAlgs(A, \thg)$; this time, with respect to certain kinds of weak equivalence in $\TAlgs$. The flexible case says that the algebras in this class are precisely the cofibrant objects of the model structure on $\TAlgs$---as was shown in~\cite[Theorem 4.12]{Lack2007Homotopy-theoretic}---whilst in the pie case, it becomes the statement that these algebras are the \emph{algebraically cofibrant} objects in the sense of~\cite{Riehl2011Algebraic}. The corresponding results for weights concern the behaviour of the limit $2$-functor $\{W, \thg\} \colon [\f J, \f K] \to \f K$ with respect to certain kinds of pointwise diagram equivalence.

In Section~\ref{sec:closure}, we study closure properties of the three classes of algebras. We prove that the semiflexibles, flexibles and pie algebras each contain the frees, and are each closed under the corresponding class of weighted colimits; this gives sufficient conditions by which to check that an algebra lies in one of these classes. In the flexible case it is known that this sufficient condition is also necessary---every flexible algebra is a flexible colimit of frees---and it is therefore natural to ask if the same is true for the semiflexible or pie algebras. For the semiflexibles, we do not completely resolve the question; but for the pie algebras, we give a much more complete answer. We will see that in general, \emph{not} every pie algebra need lie in the closure of the frees under pie colimits, but that this will be the case under certain additional assumptions, satisfied in many of our examples.

We conclude the paper in Section~\ref{sec:sf2monad} with an extended application of the results developed in it. We use it to provide a necessary and sufficient characterisation of the $2$-monads on $\Cat$ which admit a presentation of the form described above; that is, one by operations and natural transformations in which no equations are imposed between derived operations, only between derived natural transformations. We do not, in fact, consider the most general kind of $2$-monad on $\Cat$, but restrict attention to the \emph{strongly finitary} $2$-monads of~\cite{Kelly1993Finite-product-preserving}: those whose generating operations are all of the form $\f C^n \to \f C$ for some natural number $n$. The $2$-category of such $2$-monads is $2$-monadic over the product $2$-category $\Cat^\mathbb N$, and we show that  with respect to the induced $2$-monad on $\Cat^\mathbb N$, the pie algebras are precisely the strongly finitary $2$-monads which admit a presentation of the good form described above.

\section{Pie algebras for a $2$-monad}\label{sec:pie}
\looseness=-1 In this section, we introduce the notion of \emph{pie algebra} for a $2$-monad, and give an analysis showing that in many cases, the pie algebras are precisely the algebras ``free at the level of objects'', in a sense we shall make precise. We start this section, however, by recalling some necessary $2$-categorical preliminaries.

A map $f \colon X \to Y$ of a $2$-category $\f C$ is said to be \emph{fully faithful} just when for each $A \in \f C$, the functor $\f C(A, f) \colon \f C(A, X) \to \f C(A, Y)$ is fully faithful; which is to say that every $2$-cell $\alpha \colon fh \Rightarrow fk \colon A \to Y$ in $\f C$ is of the form $f \beta$ for a unique $\beta \colon h \Rightarrow k \colon A \to X$. It is easy to see that any right adjoint $2$-functor preserves fully faithfuls; and if $\f C$ is finitely complete, then we can say a little more. For then $2$-cells $\alpha$ and $\beta$ as above correspond to generalised elements of the comma object $f \downarrow f$ and the cotensor product $X^\mathbf 2$, respectively, so that $f$ is fully faithful just when the canonical arrow $X^\mathbf 2 \to f \downarrow f$ is an isomorphism. We conclude that for finitely complete $\f C$, any finite-limit-preserving $2$-functor $ \f C \to \f D$ will preserve fully faithfuls, and will moreover reflect them if it is conservative.

A map $f \colon X \to Y$ in a $2$-category $\f C$ is called an \emph{equivalence} if there exists a map $g \colon Y \to X$ and invertible $2$-cells $g \circ f \cong 1$ and $f\circ g \cong 1$. It may be that $g$ can be chosen such that $f \circ g = 1$, or such that $g \circ f = 1$; in which case we call $f$ a \emph{surjective} or an \emph{injective} equivalence, respectively. It is easy to see that $f$ is an equivalence just when it is fully faithful and admits a \emph{pseudo-section}: a map $g \colon Y \to X$ with $f \circ g \cong 1$; and similarly, that $f$ is a surjective equivalence just when it is fully faithful and admits a genuine section.

Throughout the paper, $T$ will denote a $2$-monad with rank on a complete and cocomplete $2$-category $\f C$; recall that a $2$-monad has \emph{rank} if its functor part preserves $\lambda$-filtered colimits for some regular cardinal $\lambda$. These hypotheses on $\f C$ and on $T$ will \emph{always} be assumed to be in effect, though we will repeat them for emphasis from time to time throughout the paper. Following~\cite{Blackwell1989Two-dimensional}, we write $\TAlgs$ and $\TAlg$ for the $2$-categories whose objects are, in both cases, the strict $T$-algebras, and whose morphisms are, respectively, the strict $T$-algebra morphisms, and the $T$-algebra pseudomorphisms; recall that a \emph{pseudomorphism} of $T$-algebras $(A, a) \rightsquigarrow (B, b)$ is given by a morphism $f \colon A \to B$ between the underlying
objects together with an invertible $2$-cell $\bar f \colon b \circ Tf \cong f \circ a$
satisfying two coherence axioms. We write $\fs \colon \f C \leftrightarrows \TAlgs \colon \us$ for the free-forgetful adjunction induced by the $2$-monad $T$, and will abuse notation by writing $\uu$ also for the extension of the forgetful functor to the category of pseudomorphisms $\TAlg$. Recall also the notion of \emph{pseudoalgebra} for
$T$; this being given by an object $A$ of \f C equipped with an action $a
\colon TA \to A$ which verifies compatibility with the unit and multiplication
of $T$ only up to given coherent isomorphisms. Though we shall not make direct
use of the notion of pseudoalgebra here, we shall have cause to consider the
dual notion of \emph{pseudocoalgebra} for a $2$-comonad.

We now recall from~\cite{Blackwell1989Two-dimensional} the notions of flexible and semiflexible $T$-algebra. Under our standing hypotheses on $\f C$ and on $T$, the central result of~\cite{Blackwell1989Two-dimensional}---contained in its Theorem~3.13 and Remark~3.14---assures us that the evident inclusion $2$-functor $\iota \colon \TAlgs \to \TAlg$ has a left $2$-adjoint $Q \colon \TAlg \to \TAlgs$. At an algebra $A$, the unit and
counit of the adjunction $Q \dashv \iota$ are respectively pseudo and strict morphisms of
algebras $\qun_{A} \colon A \rightsquigarrow QA$ and $p_{A} \colon QA \to A$ satisfying
the triangle equations:
\begin{equation*}
\cd[@C+1em]{
A \ar@{~>}[r]^{q_A} \ar[dr]_1 & QA \ar[d]^{p_A} \\
& A}
\qquad \text{and} \qquad
\cd[@C+1em]{
QA \ar[r]^{Qq_A} \ar[dr]_1 & Q^{2}A \ar[d]^{p_{QA}} \\
& A}
\end{equation*}
where we, as is standard, omit to write the action of the identity on objects
inclusion $\iota \colon \TAlgs \to \TAlg$. The first triangle equation asserts
that the unit $\qun_{A}$ is a section of the counit in \TAlg; in fact, as was shown in~\cite[Theorem 4.2]{Blackwell1989Two-dimensional}, $p_{A} \colon QA \to A$ is a surjective
equivalence in \TAlg with equivalence inverse $\qun_{A}$. In particular, $p_A$
is fully faithful in $\TAlg$: but since $\iota$ is locally fully faithful,
it follows that $p_A$ is also fully faithful in $\TAlgs$.

It sometimes happens that $p_{A} \colon QA \to A$ admits a strict section, a
section in \TAlgs, and in this case, the algebra $A$ is said to be
\emph{flexible}. Since $p_A$ is fully faithful, the section $a \colon A \to QA$
is an equivalence inverse of $p_{A}$, which therefore becomes a
surjective equivalence in \TAlgs: so the flexible algebras are equally those $A$
for which $p_{A}$ is a surjective equivalence in \TAlgs. This is~\cite[Theorem~4.4]{Blackwell1989Two-dimensional}, which furthermore characterises $A$ as being flexible just
when it is a retract of $QB$ for some $B$. On the other hand, it may be that $p_{A}$ does not admit a section in \TAlgs but does
admit a pseudo-section: and in this situation the algebra $A$ is said to be
\emph{semiflexible}. Since $p_A$ is fully faithful, such a
pseudo-section is an equivalence inverse for $p_A$, so that $A$
is semiflexible if and only if $p_{A} \colon QA \to A$ is an equivalence in
\TAlgs. Another characterisation, given in~\cite[Theorem~4.7]{Blackwell1989Two-dimensional}, is that
$A$ is semiflexible just when it is equivalent in \TAlgs to some $QB$.

In order to motivate the definition of pie algebra, we now provide another description of the flexible and semiflexible algebras: one given in terms of coalgebra structure for the $2$-comonad  that is induced on $\TAlgs$ by the adjunction $Q \colon \TAlg \leftrightarrows \TAlgs \colon \iota$. This comonad should rightly be called $Q \iota$, but omitting to write the $\iota$ as above, we will refer to it simply as $Q$.

\begin{Proposition}
A $T$-algebra $A$ is semiflexible just when it admits pseudo-$Q$-coalgebra structure,
and flexible just when it admits normalised pseudocoalgebra structure.
\end{Proposition}
\begin{proof}
Suppose first that $A$ is semiflexible; so we
have $a \colon A \to QA$ and an isomorphism $\alpha_{0} \colon p_{A} \circ a
\cong 1_{A}$. Since $p_A$ is fully faithful, there is a unique $2$-cell $\beta \colon a \circ p_A \cong 1_{QA}$ with $p_A \circ \beta = \alpha_0^{-1} \circ p_A$ and now
$(a,\alpha_{0}) \colon A \leftrightarrows QA \colon (p_{A},\beta)$ is an adjoint equivalence in $\TAlgs$. The $T$-algebra $QA$
underlies the cofree strict $Q$-coalgebra $(QA,\Delta_{A})$; and so by Theorem~6.1
of~\cite{Kelly2004Monoidal} we may transport this coalgebra structure along the adjoint equivalence with $A$ to obtain the required pseudo-$Q$-coalgebra structure on $A$. Direct calculation shows the induced pseudocoalgebra structure  to have coaction map $a \colon A \to QA$ and counitality constraint $\alpha_0 \colon p_A \circ a \to 1_{A}$; now if $A$ is flexible, then $\alpha_0$ may be chosen to be the identity, so that the induced pseudocoalgebra structure is normalised. This proves one direction of the proposition. For the other, if $A$ admits pseudocoalgebra structure, then it is equivalent to $QA$, and hence semiflexible; whilst if the pseudocoalgebra structure is normalised, then $A$ is a retract of $QA$ and hence flexible.
\end{proof}
With this result in mind, it is natural to define a \emph{pie algebra} to be a $T$-algebra which admits strict $Q$-coalgebra structure. As discussed in the introduction, the name is motivated by the situation where algebras are \emph{weights} $W \in [\f J, \Cat]$, for which we have the notion of pie weight, a weight which defines a pie limit. In order for the naming to be consistent, we must verify that the weights which are pie as algebras are precisely the pie weights. This is, in fact, Theorem~6.12 of~\cite{Lack2011Enhanced}; we give an alternative proof, not relying on the theory of~\cite{Lack2011Enhanced}, as Proposition~\ref{prop:pieweights} below.

We now wish to give a concrete description of the pie algebras for a range of $2$-monads of interest. Rather than doing so in an ad hoc manner, we will prove a general result which, under additional hypotheses on $T$ and on $\f C$, gives a precise and practical characterisation of the pie $T$-algebras; this result will then be employed in giving our examples.
In order to state the extra hypotheses that our result requires, we need the following notion. A morphism $f \colon X \to Y$ of a $2$-category $\f C$ is called \emph{objective} if for every fully faithful $g \colon W \to Z$ in $\f C$, the square
\begin{equation*}
\cd[@C+1em@R-0.5em]{
   \f C(Y, W) \ar[r]^{\f C(Y, g)} \ar[d]_{\f C(f, W)} &
   \f C(Y, Z) \ar[d]^{\f C(f, Z)} \\
   \f C(X, W) \ar[r]_{\f C(X, g)} &
   \f C(X, Z)
   }
\end{equation*}
is a pullback in $\Cat$. Being a pullback at the level of objects says that $g$ is orthogonal to all fully faithful morphisms in $\f C$; at the level of morphisms, it asserts a two-dimensional aspect to this orthogonality. One consequence of the definition is that any left adjoint $2$-functor preserves objectives, since any right adjoint $2$-functor preserves fully faithfuls; another is that the objectives are closed in the arrow $2$-category $\f C^\mathbf 2$ under $2$-dimensional colimits, in particular under retracts.

Due to their orthogonality, objective and fully faithful morphisms will form a factorisation system on $\f C$ so long as every morphism does in fact decompose as an objective followed by a fully faithful. This is the case, for example, when $\f C = [\f J, \Cat]$, with the objectives and  fully faithfuls being the maps which are respectively pointwise bijective on objects and pointwise fully faithful. For our characterisation result, we will suppose that $\f C$ is a $2$-category admitting (objective, fully faithful) factorisations, and that they lift to $\TAlgs$. By this we mean that $\TAlgs$ should itself admit such factorisations, with a map $f$ of $\TAlgs$ being objective or fully faithful precisely when $Uf$ is correspondingly so in $\f C$. The following result gives some alternative characterisations.
\begin{Proposition}\label{prop:whenpresbij}
If $\f C$ admits (objective, fully faithful) factorisations, then the following are equivalent:
\begin{enumerate}
\item (Objective, fully faithful) factorisations lift to $\TAlgs$;
\item $T$ preserves objective morphisms;
\item Every map of $\TAlgs$ admits a factorisation $f = g \circ h$ where $Uh$ is objective and $Ug$ is fully faithful.
\end{enumerate}
\end{Proposition}
\begin{proof}
(1) $\Rightarrow$ (2) is trivial: for if (1) holds, then $U$ preserves objectives, as does its left adjoint $F$; whence $T = UF$ will also preserve objectives. We consider next (2) $\Rightarrow$ (3). Given a $T$-algebra map $f \colon (A,a) \to (C,c)$, we factorise it in $\f C$ as an objective $h \colon A \to B$ followed by a fully faithful $g \colon B \to C$. Now in the diagram
\begin{equation*}
 \cd[@+0.5em]{TA \ar[r]^-{Th} \ar[d]_{h \circ a} & TB \ar[d]^{c \circ Tg} \ar@{-->}[dl]^b \\
  B \ar[r]_-g & C}
\end{equation*}
the top edge is objective and the bottom edge fully faithful, whence there is a unique diagonal filler $b$ as indicated. It is easy to verify that this $b$ makes $B$ into a $T$-algebra, with respect to which $g$ and $h$ are strict $T$-algebra maps. So (2) $\Rightarrow$ (3).

Finally we show that (3) $\Rightarrow$ (1). Since $U$ preserves limits and is conservative, it preserves and reflects fully faithfuls; to obtain (1), it will suffice to show that $U$ also preserves and reflects objectives. For the preservation, let $f \colon X \to Y$ be objective in $\TAlgs$. By assumption, we have a factorisation $f = gh$ with $Ug$ fully faithful and $Uh$ objective. Since $U$ reflects fully faithfuls, $g$ is fully faithful in $\TAlgs$ and so in the diagram
\begin{equation*}
 \cd[@+0.5em]{X \ar[r]^-{f} \ar[d]_{h} & Y \ar[d]^{1} \ar@{-->}[dl]^k \\
  Z \ar[r]_-g & Y}\end{equation*}
there is a unique filler $k$ as indicated. This makes $f$ a retract of $h$ in $\TAlgs^\mathbf 2$; whence $Uf$ is a retract of the objective $Uh$ in $\f C^\mathbf 2$, and thus itself objective. So $U$ preserves objectives; we must finally show that it also reflects them. Let $f \colon X \to Y$ in $\TAlgs$ with $Uf$ objective, and consider the commuting diagram:
\begin{equation*}
\cd[@C+1em]{
  FUFUX \ar@<4pt>[r]^-{FU\epsilon_X}
  \ar@<-4pt>[r]_-{\epsilon_{FUX}} \ar[d]_{FUFUf} &
  FUX \ar[r]^-{\epsilon_X} \ar[d]^{FUf} & X \ar[d]^f \\
  FUFUY \ar@<4pt>[r]^-{FU\epsilon_Y}
  \ar@<-4pt>[r]_-{\epsilon_{FUY}}  &
  FUY \ar[r]_-{\epsilon_Y} & Y\rlap{ .}}
\end{equation*}
Both rows are coequalisers in $\TAlgs$, and so the whole diagram is a coequaliser in $\TAlgs^\mathbf 2$. But since $Uf$ is objective, and both $U$ and $F$ preserve objectives, the two left-hand columns are objective, and so their coequaliser $f$ is as well. \end{proof}

\looseness=-1
As we have said, our characterisation result will assume that $\f C$ has (objective, fully faithful) factorisations which lift to $\TAlgs$; it then aims to show that the pie $T$-algebras are the algebras ``free at the level of objects''. In order to give meaning to this last phrase, we will make a further assumption on $\f C$, ensuring that every object of $\f C$ has a universal cover by one which is \emph{discrete} in a suitable sense. Then the $2$-monad $T$ will induce an ordinary monad $\Tp$ on the sub-$1$-category of $\f C$ spanned by these discrete objects, and each $T$-algebra structure on $X \in \f C$ gives rise to a $\Tp$-algebra structure on $X$'s discrete cover. Now in saying that a $T$-algebra is ``free at the level of objects'', we mean to say that this induced $\Tp$-algebra is free.

The most obvious notion of discrete object in a $2$-category is the following: we say that $X \in  \f C$ is \emph{representably discrete} when $\f C(A, X)$ is a discrete category for each $A \in \f C$. However, for our purposes it is a different notion which will be relevant. We call an object $X$ \emph{projectively discrete} if the covariant hom-functor $\f C(X, \thg)$ preserves objective morphisms. In elementary terms, this says that each morphism $f \colon X \to Z$ of $\f C$ admits a unique factorisation through each objective morphism $g \colon Y \to Z$, or more briefly: $X$ sees objective morphisms as bijective. In $\Cat$, the projectively discrete objects are precisely the discrete categories; similarly, when $\f J$ is a locally discrete $2$-category, the projectively discrete objects of $[\f J, \Cat]$ are those functors taking their image in discrete categories. In both cases, an object is projectively discrete if and only if it is representably discrete, but this need not always be so. If $\f J$ is a small $2$-category which is not locally discrete, then in $[\f J, \Cat]$, all representables are projectively discrete, but not all are representably discrete. This projective discreteness of representables is an instance of a more general phenomenon: if a right adjoint $2$-functor $\f C \to \f D$ preserves objective morphisms, then the corresponding left adjoint will preserve projective discretes. By the preceding proposition, this is so for the free/forgetful adjunction $F \colon \f C \leftrightarrows \TAlgs \colon U$ when $\f C$ admits (objective, fully faithful) factorisations and $T$ preserves objectives, so that in this case free algebras on projective discretes are again projectively discrete.

A $2$-category $\f C$ will be said to have \emph{enough \projs} if every $X \in \f C$ admits an objective morphism from a projective discrete.  It is easy to see that this is the case when $\f C = \Cat$, or when $\f C = [\f J, \Cat]$ for some small and locally discrete $\f J$; as we will explain shortly, it is in fact true of $[\f J, \Cat]$ for any small 2-category $\f J$.
 If $\f C$ is a $2$-category with enough \projs, then on choosing for each $X \in \f C$ an objective morphism $\lambda_X \colon DOX \to X$ from a projective discrete, we obtain a coreflection
 \begin{equation}\label{eq:discobj}
  \cd[@C+1.5em]{
  \Cp \ar@<4.5pt>[r]^-{D} \ar@{}[r]|-{\bot} & \ar@<4.5pt>[l]^-{O} \Co
  }
\end{equation}
with pointwise objective counit; here $\Co$ is the underlying category of $\f C$, and $\Cp$ is its full subcategory on the projective discretes. In this circumstance it is easy to see that the maps inverted by $O$ are precisely the objective morphisms. 

Note that, for example, $\Cat_d$ is the full subcategory of $\Cat$ spanned by the discrete categories; in practice, we would rather work with the equivalent category of sets and functions. 
In our examples, we will therefore allow $\Cp$ to denote a category which is merely \emph{equivalent} to the full subcategory of $\Co$ on the projective discretes; this does not change the properties of the adjunction~\eqref{eq:discobj}, these being equivalence-invariant, but does afford us the convenience of taking $\Cat_d = \Set$, or $[\f J, \Cat]_d = [\f J_0, \Set]$ for locally discrete $\f J$, and so on; henceforth, we will do so without further comment.  

In our examples, we only need to know that $[\f J,\Cat]$ has enough discretes for $\f J$  locally discrete; but in fact, as remarked above, \emph{every} presheaf $2$-category $[\f J,\Cat]$ has enough discretes, again with $[\f J, \Cat]_d = [\f J_0, \Set]$. As the theory we develop will  apply also in this situation, let us briefly justify this less obvious claim. Consider the adjunction
\begin{equation*}
\cd{
      [\f J_{0}, \Set] \ar@<4.5pt>[r]^-{D} \ar@{}[r]|-{\bot}   &
      \ar@<4.5pt>[l]^-{O} [\f J_{0}, \Cat]_{0} \ar@<4.5pt>[r]^-{F} \ar@{}[r]|-{\bot} &
      \ar@<4.5pt>[l]^-{U} [\f J, \Cat]_{0}
}
\end{equation*}
obtained by composing that which exhibits $[\f J_{0},\Set]$ as $[\f J_{0},\Cat]_{d}$ with that given by restriction and left Kan extension along the inclusion $2$-functor $\f J_{0} \to \f J$.  Clearly $U$ preserves objectives, and so $F$ preserves projective discretes; thus objects in the image of $FD$ are projectively discrete, and so we will be done if we can show the counit $\epsilon \colon FDOU \to 1$ to be pointwise objective. Since $U$ preserves and reflects objectives and $O$ inverts precisely the objectives, this is equally to show that $OU\epsilon$ is invertible, which will follow if we can show the unit $\eta \colon 1 \to OUFD$ to be invertible. Now, the component of $\eta$ at a representable $\f J_0(X, \thg)$ is the image under $O$ of the objective morphism $\f J_0(X, \thg) \to \f J(X, \thg)$ of $[\f J_0, \Cat]_0$, and so invertible; but since each of $O$, $U$, $F$ and $D$ preserve ($1$-dimensional) colimits, we conclude that \emph{every} component of $\eta$ is invertible. Thus $OU\epsilon$ is invertible, $\epsilon$ is pointwise objective, and so
$[\f J, \Cat]$ has enough discretes with $[\f J,\Cat]_{d}=[\f J_{0},\Set]$ as claimed. 

Suppose now that $\f C$ is a $2$-category with enough discretes, and $T$ a $2$-monad on it. Underlying the free-forgetful adjunction of $T$, we have the ordinary adjunction $\fo \colon \Co \leftrightarrows (\TAlgs)_0 \colon \uo$ generating the monad $\To$ on $\Co$; and quite clearly $\ToAlg = (\TAlgs)_0$. Composing~\eqref{eq:discobj} with this adjunction, we induce a monad $\Tp = OT_0 D$ on $\Cp$,
 and now have the canonical comparison functor $j \colon \ToAlg \to \TpAlg$ making both triangles commute in:
\begin{equation}\label{eq:situation}
\cd[@!C@+0.5em]{
 \ToAlg \ar[rr]^j \ar@<-3pt>[dr]_-{O\us}  & &
 \TpAlg \rlap{ .} \ar@<3pt>[dl]^-{\up} \\ &
\Cp \ar@<-3pt>[ul]_-{\fs D} \ar@<3pt>[ur]^-{\fp}}
\end{equation}
Explicitly, $j$'s action on objects is given by:
\[(a \colon TA \to A) \qquad \mapsto \qquad (OTDOA \xrightarrow{OT\lambda_A} OTA \xrightarrow{Oa} OA)\rlap{ .}\]
It is in terms of this functor $j$ that we may speak of a $T$-algebra's being ``free at the level of objects''; and we are thus ready to state our characterisation result.
\begin{Theorem}\label{thm:charthm}
Let $\f C$ have enough discretes and (objective, fully faithful) factorisations lifting to $\TAlgs$. If $\Cp$ has, and the induced monad $\Tp$ preserves, coreflexive equalisers, then a $T$-algebra $A$ is pie if and only if $jA$ is a free $\Tp$-algebra.
\end{Theorem}
The remainder of this section gives the proof of this theorem. To begin with, we assume only our standing hypotheses that $\f C$ be complete and cocomplete and that $T$ have rank. We then have the adjunction $\iota \colon \TAlgs \leftrightarrows \TAlg \colon Q$ as before, and underlying each unit map $\qun_A \colon A \rightsquigarrow QA$ a morphism $\uu\qun_A \colon \us A \to \us QA$ in $\f C$. Transposing this under the free-forgetful adjunction $\fs \colon \f C \leftrightarrows \TAlgs \colon \us$ yields a map in $\TAlgs$, which we denote by $\rho_A \colon \fs \us A \to QA$.
\begin{Lemma}\label{lem:rhoobj}
Each $\rho_A$ is an objective morphism of $\TAlgs$.
\end{Lemma}
\begin{proof}
One way of proving this would be to analyse $Q$ in terms of isocodescent objects, as was done in~\cite{Lack2002Codescent}; however, it is just as easy to proceed directly. For the one-dimensional aspect of $\rho_A$'s objectivity, we must show that for any square \begin{equation*}
\cd{
  \fs \us A \ar[r]^-{\rho_A} \ar[d]_h & QA \ar@{-->}[dl] \ar[d]^{k} \\
  C \ar[r]_-f & D
}
\end{equation*}
in $\TAlgs$ with $f$ fully faithful, there is a unique map $QA \to C$ as indicated making both induced triangles commute. Equivalently, after transposing under adjunction, we must show that for any pseudomorphism $(k, \bar k) \colon A \rightsquigarrow D$, each factorisation of $k \colon \us A \to \us D$ through $\us f$ is the image under $\uu$ of a unique factorisation of $(k, \bar k)$ through $f$. Thus given $k =  \us f \circ h$ we seek a unique $\bar h$ such that $(h, \bar h) \colon A \rightsquigarrow C$ is a pseudomorphism with $(k, \bar k) = f \circ (h, \bar h)$. This last condition asserts a pasting equality
\begin{equation*}
\cd[@C+1em]{
  \us\fs\us A \ar[r]^{\us\fs k} \dtwocell{dr}{\bar k} \ar[d]_{U\epsilon_A} &  \us\fs \us D \ar[d]^{U\epsilon_D} \\
  \us A \ar[r]_{k} & \us D
} \qquad = \qquad
\cd[@C+1em]{
  \us\fs\us A \ar[r]^{\us\fs h} \ar[d]_{U\epsilon_A} \dtwocell{dr}{\bar h} & \us\fs\us C \ar[r]^{\us\fs\us f} \ar[d]|{U\epsilon_C} \twoeq{dr} & \us\fs \us D \ar[d]^{U\epsilon_D} \\
  \us A \ar[r]_{h} &  \us C \ar[r]_{\us f} & \us D\rlap{ ;}
}
\end{equation*}
but as $f$ is fully faithful in $\TAlgs$, $\us f$ is fully faithful in $\f C$, whence there is exactly  one $2$-cell $\bar h$ for which this is so. It is now straightforward to check that this $\bar h$ does indeed make $(h, \bar h)$ into a pseudomorphism $A \rightsquigarrow C$. This verifies the one-dimensional aspect of $\rho_A$'s objectivity, and the two-dimensional aspect follows similarly.
\end{proof}
\begin{Lemma}\label{cor:comonadmap}
The $\rho_A$'s are the components of a comonad morphism $\rho \colon \fs \us \to Q$; moreover, $Q$ is, up to isomorphism, the unique comonad on $\TAlgs$ that has pointwise fully faithful counit and admits a pointwise objective comonad morphism from $FU$.
\end{Lemma}
\begin{proof}
It is easy to see that $\rho$ is $2$-natural, since $q$ is, and also that the axiom $p \circ \rho = \epsilon \colon \fs \us \to 1$ expressing compatibility with the counits of $Q$ and $FU$ is satisfied. We must further show that $\Delta \circ \rho = \rho \rho \circ \fs \eta \us$, where $\Delta = Qq$ is the comultiplication of $Q$. The map $\Delta\circ \rho$ corresponds under adjunction to $\us\Delta\circ \uu\qun = \us Q\qun \circ \uu\qun = \uu\qun Q \circ \uu\qun$, but also $\rho \rho \circ  \fs \eta \us = \rho Q \circ  \fs \us \rho \circ \fs \eta \us$ corresponds to $\uu\qun Q\circ  \us \rho\circ  \eta \us = \uu \qun Q\circ \uu q$, so that $\Delta \circ \rho = \rho \rho \circ \fs \eta \us$ as required.

Now given any comonad $(R, e, d)$ on $\TAlgs$ which has pointwise fully faithful counit $e$ and admits a pointwise objective comonad morphism $\sigma \colon FU \to R$, we have a commutative square in $[\TAlgs, \TAlgs]$ as on the left of
\begin{equation*}
\cd[@+1em]{
FU \ar[r]^\rho \ar[d]_\sigma & Q \ar[d]^p \ar@{-->}[dl]^{\alpha} \\
R \ar[r]_e & 1} \qquad \qquad
\cd[@C+1em@R-1em]{
  \fs \us \ar[r]^-{\rho} \ar[d]_{\fs \eta \us} &
  Q \ar@{-->}[ddl] \ar[dd]^\alpha \\
  \fs \us \fs \us \ar[d]_{\sigma \sigma} \\
  RR \ar[r]_{eQ} & R\rlap{ .}
}
\end{equation*}
Both sides of this square being (objective, fully faithful) factorisations of $\epsilon$, there is a unique isomorphism $\alpha \colon Q \to R$ as indicated making both triangles commute. By construction, this $\alpha$ is compatible with the counits of $Q$ and $R$; whilst compatibility with the comultiplications follows by observing that the square on the right above has, by orthogonality, a unique diagonal filler, and that both $d \circ \alpha$ and $\alpha\alpha \circ \Delta$ are such fillers.
\end{proof} 

In the statement of the following result, we write $\ToAlg /_\mathrm{ff} A$ for the full subcategory of $\ToAlg / A$ spanned by the fully faithful maps into $A$.
\begin{Lemma}
If $\f C$ has enough \projs, and $T$ preserves objectives, then for each $T$-algebra $A$, the functor
\begin{equation*}
    j_A \colon \ToAlg /_\mathrm{ff} A \to \TpAlg / jA\rlap{ ,}
\end{equation*}
induced by slicing the comparison functor $j$ of~\eqref{eq:situation}, is fully faithful.
\end{Lemma}
\begin{proof}
Given a $T$-algebra $(A,a)$ and objects $(B, b) \xrightarrow m (A,a)$ and $(C, c) \xrightarrow n (A,a)$   of $\ToAlg /_\mathrm{ff} (A,a)$, we must show that each map $j_A(m) \to j_A(n)$ in $\TpAlg / jA$ is the image under $j_A$ of a unique map $m \to n$. To give a map $j_A(m) \to j_A(n)$ is to give a map $g \colon OB \to OC$ in $\Cp$ making the left-hand diagram below commute:
\begin{equation*}
  \cd{
  	OTDOB \ar[rr]^{OTDg} \ar[d]_{OT\lambda_B} & & OTDOC \ar[d]^{OT\lambda_C} \\
	OTB \ar[d]_{Ob} & & OTC \ar[d]^{Oc} \\
	OB \ar[dr]_{Om} \ar[rr]_{g} && OC \ar[dl]^{On} \\ &
	OA
  }\qquad \quad
      \cd[@+1em]{
        DOB \ar[d]_{\lambda_{C} \circ Dg} \ar[r]^-{\lambda_{B}} &
        B \ar[d]^{m} \ar@{-->}[dl]_k
        \\  C \ar[r]_{n} & A\rlap{ .}
    }
\end{equation*}
Now the commutativity of the lower triangle on the left is equivalent to that of the right-hand square in $\f C_0$; and since $\lambda_{B}$ is objective and $n$ is fully faithful, there is a unique diagonal filler $k$ as indicated. Commutativity of the upper triangle so induced asserts that $\lambda_{C} \circ Dg = k \circ \lambda_{B} = \lambda_{C} \circ DOk$ which, transposing through the adjunction, is equally well the assertion that $g = Ok$. Thus it will follow that $k$ is the desired unique lifting of $g$ so long as we can show that it is in fact a $T$-algebra map $(B,b) \to (C,c)$, or in other words, that $k \circ b = c \circ Tk \colon TB \to C$.

By orthogonality, it will suffice to verify this equality on postcomposition with a fully faithful morphism and on precomposition with an objective one.
On the one hand, we have the fully faithful $n \colon C \to A$, and calculate that $n \circ k \circ b = m \circ b = a \circ Tm = a \circ Tn \circ Tk = n \circ c \circ Tk$ as required. On the other, we have the objective $\lambda_{TB} \colon DOTB \to TB$, and showing that $k \circ b \circ \lambda_{TB} = c \circ Tk \circ \lambda_{TB}$ is equivalent to showing that $Ok \circ Ob = Oc \circ OTk$. Since $Ok = g$, we have from the left-hand commutative diagram above that $Ok \circ Ob \circ OT\lambda_B = Oc \circ OT\lambda_C \circ OTDOk = Oc \circ OTk \circ OT\lambda_B$. But since $\lambda_B$ is objective, so also is $T\lambda_B$, whence $OT\lambda_B$ is invertible; so that $Ok \circ Ob = Oc \circ OTk$ as required.
\end{proof}
Using this, we may now prove a direct predecessor of our main Theorem~\ref{thm:charthm}:
\begin{Proposition}\label{prop:pieproj}
If $\f C$ has enough discretes and (objective, fully faithful) factorisations lifting to $\TAlgs$, then a $T$-algebra $A$ is pie if and only if $jA$ admits $\fp \up$-coalgebra structure.
\end{Proposition}
\begin{proof}
We first observe that $j \colon \ToAlg \to \TpAlg$ inverts objective morphisms. 
Indeed, if $f$ is objective in $\ToAlg$, then $Uf$
is objective in $\f C_0$, whence $OUf$ is invertible in $\Cp$. But $OUf = U_d jf$ and $U_d$ is conservative, and so $jf$ is invertible in $\TpAlg$ as claimed.

Now let $Q_0$ denote the comonad on $\ToAlg$ underlying the $2$-comonad $Q$. By Lemmas~\ref{lem:rhoobj} and~\ref{cor:comonadmap}, we have a  comonad morphism $\rho_0 \colon \fo \uo \to Q_0$ which is pointwise objective; since $j$ inverts objectives, we obtain from this a natural isomorphism $j \rho_0 \colon j \fo \uo \to j Q_0$. We also have a comonad morphism $\fo \lambda \uo \colon \fo DO \uo \to \fo \uo$, which is again pointwise objective, since $\lambda$ is so and $\fo$ preserves objective morphisms, and thus we obtain a natural isomorphism $j \fo \lambda \uo \colon j\fo DO \uo \to j\fo \uo$. Now by composing inverses we obtain a natural isomorphism
\[
  \phi := jQ_0 \xrightarrow{\ \ \ (j\rho_0)^{-1}\ \ \ } j \fo \uo \xrightarrow{\ \ \ (j \fo \lambda \uo)^{-1}\ \ \ }
j\fo DO \uo = \fp \up j\]
which may be shown to equip $j$ with the structure of a comonad morphism $(\ToAlg, Q_0) \to (\TpAlg, \fp \up)$. Thus if $A$ is a pie algebra---hence admitting $Q$-coalgebra structure---then $jA$ is a $\fp \up$-coalgebra. Conversely, suppose that $jA$ is a $\fp \up$-coalgebra, with structure map $g \colon jA \to \fp \up jA$. Then the left-hand diagram commutes in:
\begin{equation*}
  \cd[@!C]{jA \ar[dr]_{j(1)} \ar[r]^-{g} & \fp \up j A \ar[d]|{\epsilon_{jA}} \ar[r]^-{\phi^{-1}} & jQA \ar[dl]^{j(p_A)} \\ &
  jA}
  \qquad \quad
  \cd[@!C]{A \ar[dr]_{1} \ar[rr]^-{a} & & QA\rlap{ .} \ar[dl]^{p_A} \\ &
  A}
\end{equation*}
Since both $1 \colon A \to A$ and $p_A \colon QA \to A$ are fully faithful $T$-algebra maps, we have, by the preceding lemma, a $T$-algebra morphism $a \colon A \to QA$ as on the right with $j(a) = \phi^{-1} \circ g$. We will show that this $a$ satisfies the comultiplication axiom, so making $A$ into a $Q$-coalgebra and hence a pie algebra. So observe that both triangles in the diagram:
\begin{equation*}
  \cd[@!C]{A \ar[dr]_{1} \ar@<3.5pt>[rr]^-{\Delta_A \circ a} \ar@<-3.5pt>[rr]_-{Qa \circ a} & & QQA \ar[dl]^{pp_A} \\ &
  A}
\end{equation*}
are commutative; moreover, both diagonal maps are fully faithful and so to show that the two horizontal maps are equal, as we must, it is enough by the preceding lemma to do so after applying $j$: which follows by an entirely straightforward calculation.
\end{proof}
Finally, we obtain:
\begin{proof}[Proof of Theorem~\ref{thm:charthm}]
Comparing this theorem's statement with that of Proposition~\ref{prop:pieproj}, we see that the gap between the two is the gap between the free $\Tp$-algebras and the $\fp \up$-coalgebras. This gap is certainly closed if the adjunction $\fp \dashv \up$ is comonadic as well as monadic, but this is more than is needed. A necessary and sufficient condition is that for each coalgebra $a \colon A \to \fp \up A$, the parallel pair $(\eta_{\up A}, \up a) \colon \up A \rightrightarrows \up \fp\up A$ in $\Cp$ should admit an equaliser which is preserved by $\fp$; for then it follows that the equaliser map $E \rightarrowtail \up A$ exhibits $A$ as free on $E$. This condition will certainly be satisfied if the category $\Cp$ has, and $\fp$ preserves, all coreflexive equalisers; but since $\up$ creates limits, this is equally to ask, as in the statement of this theorem, that $\Cp$ have and $\Tp$ preserve coreflexive equalisers.\end{proof}

\section{Examples of pie algebras}\label{sec:exs}
We now apply the results developed in the previous section to the analysis of the pie algebras for some $2$-monads of interest. In each of our examples, we shall verify the conditions of Theorem~\ref{thm:charthm} and use it to give an explicit description of the class of pie algebras.

\subsection{Categories with structure}\label{sec:catexs}
As we commented in the introduction, many natural examples of $2$-monads reside on the $2$-category $\Cat$, their algebras being certain kinds of structured category. We have already seen that $\Cat$ has enough discretes and (objective, fully faithful) factorisations, and for any given $2$-monad on $\Cat$ it is relatively straightforward to verify whether the remaining  hypotheses of Theorem~\ref{thm:charthm} are satisfied. We consider in detail one example: the $2$-monad $T$ for strict monoidal categories.

First observe that if $F \colon \f A \to \f C$ is a strict monoidal functor between strict monoidal categories, and $F = GH$ a (bijective on objects, fully faithful) factorisation of its underlying functor, then the interposing category $\f B$ admits a strict monoidal structure with respect to which both $G$ and $H$ are strict monoidal; the reason that this is possible is that the arities required to describe strict monoidal structure, the discrete categories $0,1,2$ and $3$, are projectively discrete. It therefore follows from Proposition~\ref{prop:whenpresbij} that (objective, fully faithful) factorisations lift to $\TAlgs$.

Now consider the induced monad $\Tp$ on $\Cat_d = \Set$. This sends a set $X$ to the set of objects of the free strict monoidal category on $X$; it is thus the monoid monad on $\Set$, and the comparison functor $j \colon \ToAlg \to \TpAlg$ sends a strict monoidal category to its underlying monoid of objects. Finally, because the monoid monad $\sum_{n \in \mathbb N} (\thg)^n$ is a coproduct of representable functors, it preserves all connected limits, and so in particular coreflexive equalisers. We therefore conclude from Theorem~\ref{thm:charthm} that:

\begin{Proposition}
  A strict monoidal category is pie just when its underlying monoid of objects is free; that is, just when it is a (many-sorted) PRO in the sense of~\cite{Mac-Lane1965Categorical}.
\end{Proposition}
Arguments of an entirely analogous form characterise the pie algebras for many other $2$-monads on $\Cat$; we list a few below as a representative sample.
\begin{itemize}
\item \emph{Symmetric strict monoidal categories}: (objective, fully faithful) factorisations again lift to $\TAlgs$, and the induced monad $\Tp$ on $\Set$ is again the monoid monad. Thus pie algebras are symmetric strict monoidal categories whose underlying monoid of objects is free; in other words, they are (many-sorted) PROPs in the sense of~\cite{Mac-Lane1965Categorical}.\vskip0.5\baselineskip
\item \emph{Categories with strictly associative and unital finite products}: arguing exactly as in the preceding two examples, we conclude that in this case the pie algebras are (many-sorted) Lawvere theories.\vskip0.5\baselineskip
\item \emph{(Symmetric) monoidal categories}: this time the induced monad on $\Set$ is the monad for \emph{pointed magmas}: these being sets equipped with a binary and a nullary operation, satisfying no laws. The underlying functor of this monad, like that of any monad free on a signature, is a coproduct of representables, and so preserves coreflexive equalisers as before; thus by Theorem~\ref{thm:charthm}, a (symmetric) monoidal category is pie just when its underlying pointed magma of objects is free.\vskip0.5\baselineskip
\end{itemize}

This last example was particularly easy to analyse by virtue of the induced monad $T_d$'s being free on a signature. This property is shared by many other $2$-monads on $\Cat$, including those whose algebras are categories with finite products, or distributive categories, or categories equipped with a monad, or categories with finite products and monoidal structure, or linearly distributive categories, and so on.  We will consider such $2$-monads in more detail in Section~\ref{sec:sf2monad} below; in particular, Proposition~\ref{prop:stronglypiepiealgs} shows that the above characterisation of the pie monoidal categories is equally valid for the algebras of any such $2$-monad.

\subsection{Weights}Given some small $2$-category $\f J$, we may consider the adjunction $F \colon [\mathrm{ob}\ \f J, \Cat] \leftrightarrows [\f J, \Cat] \colon U$ given by restriction and left Kan extension along the inclusion of objects $\mathrm{ob}\ \f J \to \f J$. This adjunction is strictly $2$-monadic, so allowing us to identify $[\f J, \Cat]$ with the $2$-category of algebras for the induced $2$-monad $T = UF$. As in the introduction, objects of $[\f J, \Cat]$ are to be thought of as weights for limits or colimits; and our goal is to show that a weight $W$ is pie as a $T$-algebra just when it determines a \emph{pie limit}---that is, one constructible from products, inserters and equifiers. The pie limits were the primary object of study of~\cite{Power1991A-characterization}, and Corollary 3.3 of that paper shows that a weight $W \in [\f J, \Cat]$ determines a pie limit just when the presheaf
\begin{equation*}
\f J_0 \xrightarrow{W_0} \Cat_0 \xrightarrow{\textrm{ob}} \Set
\end{equation*}
is a coproduct of representables. Our goal, then, is to show that a weight $W$ has this property just when it is pie as a $T$-algebra. As mentioned before, this was done in~\cite[Theorem 6.12]{Lack2011Enhanced}; however, it is quite straightforward to give an alternative proof using our Theorem~\ref{thm:charthm}. 

Observe first that $[\mathrm{ob}\ \f J, \Cat]$ has (objective, fully faithful) factorisations which lift to $\TAlgs = [\f J, \Cat]$, as in both $2$-categories a morphism is objective or fully faithful just when it is pointwise so. As for the induced monad $\Tp$ on $[\mathrm{ob}\ \f J, \Cat]_{d}=[\mathrm{ob}\ \f J, \Set]$ we may calculate this to be given by the formula
\begin{equation*}
  (\Tp X)(j) = \textstyle\sum_{k \in \f J}\, \f J_0(k, j) \times Xk\rlap{ .}
\end{equation*}
This shows $\Tp$ to be pointwise a coproduct of representable functors; as such it preserves connected limits and in particular coreflexive equalisers. It furthermore shows that $\TpAlg \cong [\f J_0, \Set]$, and now under this identification, the comparison functor $\ToAlg \to \TpAlg$ becomes the functor $[\f J, \Cat]_0 \to [\f J_0, \Set]$ sending $W \colon \f J \to \Cat$ to the  presheaf $\mathrm{ob} \circ W_0$.
We therefore conclude from Theorem~\ref{thm:charthm} that:
\begin{Proposition}\label{prop:pieweights}
A weight $W \colon \f J \to \Cat$ is a pie algebra for the weight $2$-monad on $[\mathrm{ob}\ \f J, \Cat]$ just when its underlying presheaf $\mathrm{ob} \circ W_0 \colon \f J_0 \to \Set$ is free, i.e., a coproduct of representables; that is, just when $W$ determines a pie limit.
\end{Proposition}
Let us remark that our proof of the above proposition was not self-contained but required an application of Corollary 3.3 of~\cite{Power1991A-characterization}. In fact, we could equally have appealed to our Proposition~\ref{thm:1in2} below, which can be seen as a generalisation of Power and Robinson's result from weights to more general algebras.

\subsection{$2$-categories and related structures}\label{sec:2catsrelated}
For our third example, we consider $2$-categories and related structures as algebras for a $2$-monad, and determine which of these algebras are pie. We begin with the case of $2$-categories.

By a \emph{category-enriched graph}, we mean a diagram $A_1 \rightrightarrows A_0$ of categories with $A_0$ discrete. We write $\f P$ for the parallel pair category $\bullet \rightrightarrows \bullet$, and write $\CatGph$ for the full and locally full sub-$2$-category of $[\f P, \Cat]$ spanned by the category-enriched graphs. On the mere category $[\f P, \Cat]_0$, we have the monad whose algebras are categories internal to $\Cat$---thus double categories---and  it is easy to enrich this to a $2$-monad $S$ on $[\f P, \Cat]$, and then to restrict this to a $2$-monad $T$ on $\CatGph$.

The algebras for $T$ are $2$-categories; the strict algebra morphisms are $2$-functors, and the pseudomorphisms are pseudofunctors. As for the algebra $2$-cells, these are the \emph{icons} of~\cite{Lack20082-nerves,Lack2010Icons}; recall that an icon between two $2$-functors $G, H \colon A \rightrightarrows B$ is an oplax natural transformation $\alpha \colon G \Rightarrow H$ whose $1$-cell components are all identities. We write $\Icon$ for the $2$-category of $2$-categories, $2$-functors and icons; it is thus $\TAlgs$ for the above-defined $T$, and our task is to identify the pie $T$-algebras.

It turns out that the results of Theorem~\ref{thm:charthm} are not directly applicable in this situation, as although $\CatGph$ has (objective, fully faithful) factorisations, these do not lift to $\Icon$. We must therefore follow a more roundabout route to our characterisation. We will reduce the problem to that of characterising the pie $2$-categories with a fixed object set $X$, seen as algebras for a $2$-monad on category-enriched graphs with object set $X$; the point being that in this situation, (objective, fully faithful) factorisations \emph{do} lift, with the other hypotheses of Theorem~\ref{thm:charthm} being similarly satisfied.

Thus given some set $X$, we let $\Icon_{X}$ denote the locally full sub-$2$-category of \Icon spanned by the $2$-categories with object set $X$ and the identity on objects 2-functors between them. There is a corresponding sub-$2$-category $\CatGph_X$ of $\CatGph$, and the free-forgetful adjunction $F \colon \CatGph \leftrightarrows \Icon \colon U$ restricts to an adjunction $F_X \colon \CatGph_X \leftrightarrows \Icon_X \colon U_X$ which is again strictly $2$-monadic. If we call the induced $2$-monad $T_X$, then our claim is that a $2$-category $A$ with object set $X$ is pie as a $T$-algebra just when it is pie as a $T_X$-algebra.  This will follow from:

\begin{Lemma}\label{lem:qicon}
The comonad $Q$ on $\Icon$ associated with $T$ may be so chosen as to restrict to $\Icon_X$ for each $X$, there yielding the comonad $Q_X$ associated with $T_X$. 
\end{Lemma}

For indeed, if the above lemma holds then certainly each $Q_{X}$-coalgebra yields a $Q$-coalgebra structure.  Conversely, if $A \in \Icon_{X}$ admits a $Q$-coalgebra structure $a \colon A \to QA$ then $p_{A}$'s being the identity on objects ensures that so too is $a$ which, as a morphism of $\Icon_{X}$, thereby exhibits $A$ as a $Q_{X}$-coalgebra.

In proving this lemma, we will need a characterisation of the fully faithful maps both in $\Icon$ and $\Icon_{X}$, and some understanding of the objectives. With regard to $\Icon$, note that both the forgetful $2$-functor $\Icon \to \CatGph$ and the inclusion $\CatGph \to [\f P, \Cat]$ are limit-preserving and conservative, whence a $2$-functor $G \colon A \to B$ is fully faithful in $\Icon$ just when its underlying morphism
\begin{equation}\label{eq:catgphmor}
\cd{
  A_1 \ar@<-4pt>[d] \ar@<4pt>[d] \ar[r]^{G_1} &
  B_1 \ar@<-4pt>[d] \ar@<4pt>[d] \\
  A_0 \ar[r]_{G_0} & B_0
 }
\end{equation}
is so in 
$[\f P, \Cat]$. This happens just when $G_1$ is a fully faithful functor, and $G_0$ an injective function, so that the fully faithful maps in $\Icon$ are the $2$-functors which are injective on objects and locally fully faithful.

In the case of $\Icon_{X}$ we have the limit-preserving and conservative forgetful $2$-functor $\Icon_{X} \to \CatGph_{X}$, so that a morphism $G \colon A \to B$ is fully faithful in $\Icon_X$ just when its underlying morphism is so in $\CatGph_{X}$.  This latter 2-category is equivalent to $[X \times X,\Cat]$, and from our understanding of the fully faithful morphisms therein, we deduce that  $G$ is fully faithful in $\Icon_{X}$ just when its underlying morphism~\eqref{eq:catgphmor} has $G_{1}$ fully faithful; which is to say that $G$ is a locally fully faithful $2$-functor.

Now $\Icon_X$ admits a (bijective on $1$-cells, locally fully faithful) factorisation system; since the right class comprises the fully faithfuls in $\Icon_X$, it follows that the left class comprises the objectives. Considering the inclusion $\Icon_X \to \Icon$, it is immediate that it preserves fully faithfuls; but it also preserves objectives, since in $\Icon$, a $2$-functor which is the identity on objects and $1$-cells is orthogonal to any locally fully faithful $2$-functor, and so certainly objective. Let us take this opportunity also to remark that the objectives in $\CatGph_X$ are again the morphisms bijective on $1$-cells, so that (objective, fully faithful) factorisations in $\Icon_X$ are in fact lifted from $\CatGph_X$.

\begin{proof}[Proof of Lemma~\ref{lem:qicon}]
If $A$ is a $2$-category with object set $X$, then by Lemma~\ref{cor:comonadmap}, we have an (objective, fully faithful) factorisation
\begin{equation}\label{eq:factorise1}
FUA \xrightarrow{\epsilon_A} A \qquad = \qquad FUA \xrightarrow{\rho_A} QA \xrightarrow{p_A} A\rlap{ ,}
\end{equation}
where $\epsilon$ is the counit of the comonad $FU$ on $\Icon$. Since $\epsilon_A \colon FUA \to A$   is the identity on objects, it may be viewed as a map of $\Icon_X$; seen as such, it is in fact the counit component at $A$ of the comonad $F_XU_X$ on $\Icon_X$, and we therefore have in $\Icon_X$ an (objective, fully faithful) factorisation 
\begin{equation}\label{eq:factorise2}
FUA \xrightarrow{\epsilon_A} A \qquad = \qquad FUA \xrightarrow{(\rho_X)_A} Q_XA \xrightarrow{(p_X)_A} A\rlap{ .}
\end{equation}
The above discussion shows this also to be an (objective, fully faithful) factorisation in $\Icon$, and we conclude that $QA$ is uniquely isomorphic to $Q_XA$ for each object $A \in \Icon_X$. Transporting along these unique isomorphisms, we may take it that in fact $QA = Q_XA$, $\rho_A = (\rho_X)_A$ and $p_A = (p_X)_A$ for each $A \in \Icon_X$.

On doing so, we have the endofunctor $Q$ and the counit $p \colon Q \to 1$ both the identity on objects; because $pQ \circ \Delta = 1$, it follows from this that the comultiplication $\Delta \colon Q \to QQ$ is also the identity on objects, whence the comonad $Q$ restricts from $\Icon$ to $\Icon_X$ for each set $X$. But by construction, the factorisation~\eqref{eq:factorise1} restricts to give the factorisation~\eqref{eq:factorise2}, so that by Lemma~\ref{cor:comonadmap}, the restricted comonad is indeed $Q_{X}$.  
 \end{proof}

Using this lemma, our problem is thus reduced to that of characterising the pie $T_X$-algebras.  As previously noted the (objective, fully faithful) factorisations of $\CatGph$ lift to $\Icon_{X} \cong \textnormal{T}_X\textnormal{-Alg}_\mathrm{s}$ and since $\CatGph_X$ is equivalent to $[X \times X, \Cat]$ it has enough discretes.  We may take $(\CatGph_X)_d$ to be $[X \times X, \Set]$: now the induced monad $(T_X)_d$ is the monad for categories with object set $X$, and the comparison functor $j \colon (T_X)_0\text{-Alg} \to (T_X)_d\text{-Alg}$ sends a $2$-category with object set $X$ to its underlying ordinary category. To see that $(T_X)_d$ preserves coreflexive equalisers, observe that it is given by 
\[
  (T_X)_d(A)(x,y) = \sum_{x = x_0, \dots, x_n = y\in X} A(x_0, x_1) \times \dots \times A(x_{n-1}, x_n)\rlap{ .}
\]
It is thus pointwise a coproduct of representables and so preserves connected limits, in particular coreflexive equalisers. We conclude from Theorem~\ref{thm:charthm} that:
\begin{Proposition}
A $2$-category is pie just when its underlying ordinary category is free on a graph.
\end{Proposition}
The arguments given above may be generalised in two directions. The first is to consider bicategories in place of $2$-categories. In this case, we begin by considering the $2$-category of bicategories, strict homomorphisms and icons, seeing that this is again $2$-monadic over $\CatGph$, and now continue the argument as before. Our conclusion will differ in the nature of the underlying structure which is required to be free. As in~\cite{Lack2004A-Quillen}, we define a \emph{compositional graph} to be a directed graph $X_1 \rightrightarrows X_0$ equipped with identity and composition operations, subject to no axioms: thus a category is a compositional graph in which the unit and associativity laws are verified. The objects and $1$-cells of any bicategory form a compositional graph, the underlying compositional graph of the bicategory. Now our result is that:
\begin{Proposition}
A bicategory is pie if and only if its underlying compositional graph is free.
\end{Proposition}

The second direction of generalisation concerns double categories. This generalisation is in fact  also a simplification, as in this case Theorem~\ref{thm:charthm} will be immediately applicable. Recall from above that we formed the $2$-category $\CatGph$ as a full sub-$2$-category of $[\f P, \Cat]$, and the $2$-monad $T$ for $2$-categories thereon as the restriction and corestriction of a $2$-monad $S$ on $[\f P, \Cat]$. This latter $2$-monad we constructed as a $2$-dimensional enrichment of the free category monad on $[\f P, \Cat]_0$; as such the strict $S$-algebras are categories internal to $\Cat$, that is, double categories. The corresponding strict and pseudo algebra morphisms are \emph{double functors} and \emph{pseudo double functors}, respectively, whilst the algebra $2$-cells are \emph{horizontal transformations}; our terminology throughout is that of~\cite{Grandis1999Limits}. 

The 2-category $[\f P,\Cat]$ on which $S$ resides has (objective, fully faithful) factorisations;
what distinguishes this case from that of $2$-categories is that these factorisations lift. The easiest way of seeing this is to observe that a strict double functor $G$, when seen as an internal functor in $\Cat$, admits a factorisation
\begin{equation*}
\cd{
  A_1 \ar@<-4pt>[d] \ar@<4pt>[d] \ar[r]^{G_1} &
  B_1 \ar@<-4pt>[d] \ar@<4pt>[d] \\
  A_0 \ar[r]_{G_0} & B_0
 } \qquad = \qquad 
\cd{
  A_1 \ar@<-4pt>[d] \ar@<4pt>[d] \ar[r]^{H_1} &
  C_1 \ar@<-4pt>[d] \ar@<4pt>[d] \ar[r]^{K_1} &
  B_1 \ar@<-4pt>[d] \ar@<4pt>[d] \\
  A_0 \ar[r]_{H_0} & 
  C_0 \ar[r]_{K_0} & B_0
 }\end{equation*}
into a pair of internal functors for which $H_0$ and $H_1$ are bijective on objects, and $K_0$ and $K_1$ are fully faithful; now by applying Proposition~\ref{prop:whenpresbij}, the claim follows. The 2-category $[\f P,\Cat]$ has enough discretes, with $[\f P,\Cat]_{d}=[\f P,\Set]$, the category of graphs; the induced monad $S_d$ on $[\f P,\Set]$ is the free category monad, and the comparison functor $j \colon \textnormal{S}_0\textnormal{-Alg} \to \textnormal{S}_d\textnormal{-Alg}$ assigns to each double category its underlying category of objects and vertical arrows.  Finally, as the free category monad is pointwise a coproduct of representables, it preserves connected limits and in particular coreflexive equalisers; thus we conclude from Theorem~\ref{thm:charthm} that:
\begin{Proposition}
A double category is pie if and only if its underlying vertical category is free on a graph.
\end{Proposition}
We may combine the two generalisations given above to obtain a fourth characterisation result for the \emph{pseudo double categories} of~\cite{Grandis1999Limits}; it states that a pseudo double category is pie just when its underlying compositional graph of objects and vertical arrrows is free.

\section{Characterisation theorems for algebras and weights}\label{sec:char}
We now turn to the second main objective of the paper, that of clarifying the relationship between the semiflexible, flexible and pie algebras for a $2$-monad; we do so under our standing assumptions that the $2$-monad in question should have rank and should reside on a complete and cocomplete $2$-category. We prove three theorems, each describing a property which, when given in its weakest form, characterises completely the semiflexible algebras; when strengthened slightly, characterises the flexible algebras; and when strengthened further still, characterises the pie algebras.

In fact, we pay special attention to the case of the semiflexible, flexible and pie weights, and so each of our characterisation theorems will have a second form dealing with these. So as to be able to give this second form, we  now establish some notational conventions concerning weighted limits. For any weight $W \in [\f J, \Cat]$ and any $2$-category $\f K$, we have a $2$-functor
\begin{align*}
   \mathrm{Cone}_W \colon \f K^\op \times [\f J, \f K] & \to \Cat \\
   (X, D) & \mapsto [\f J, \Cat](W,\, \f K(X, D\thg))\rlap{ ;}
\end{align*}
we call an object $\alpha \in \mathrm{Cone}_W(X,D)$ a \emph{$W$-weighted cone from $X$ to $D$}, and denote it by $ \alpha \colon X \dot{\rightarrow} D$. We write $\beta \circ \alpha$ and $\alpha \circ f$ for its postcomposition with a $2$-natural transformation $\beta \colon D \to D'$ and its precomposition with a map $f \colon X' \to X$. For some given $D \colon \f J \to \f K$, a \emph{limit of $D$ weighted by $W$} is a representation for the $2$-functor $\mathrm{Cone}_W(\thg, D) \colon \f K^\op \to \Cat$. Of course, if the $2$-category $\f K$ admits $W$-weighted limits for all $D \in [\f J, \f K]$, then on making a choice of such, we obtain a $2$-functor $\{W, \thg\} \colon [\f J, \f K] \to \f K$.

Given $W \in [\f J, \Cat]$ and a $2$-category $\f K$, we can also consider the $2$-functor
\begin{align*}
   \mathrm{PsCone}_W \colon \f K^\op \times \Ps(\f J, \f K) & \to \Cat \\
   (X, D) & \mapsto \Ps(\f J, \Cat)(W,\, \f K(X, D\thg))\rlap{ ,}
\end{align*}
whose elements $\alpha \in \mathrm{PsCone}_W(X,D)$ we call \emph{$W$-weighted pseudocones} $\alpha \colon X \dot{\rightsquigarrow} D$; now given some $D \colon \f J \to \f K$, a \emph{pseudolimit of $D$ weighted by $W$} is a representation for $\mathrm{PsCone}_W(\thg, D)$. If the $2$-category $\f K$ admits $W$-weighted pseudolimits for all $D \colon \f J \to \f K$, then on making a choice of such we obtain a $2$-functor $\{W, \thg\}_{\mathrm{ps}} \colon \Ps(\f J, \f K) \to \f K$.

\subsection{First characterisation}
As mentioned in the introduction, the semiflexible algebras are characterised by the property that  any algebra pseudomorphism out of one may be replaced by an isomorphic strict morphism; this was proven in~\cite[Theorem 7.2]{Blackwell1989Two-dimensional}. What we will now show is that the flexibles and pies may be further distinguished by the increasingly well-behaved manner in which this replacement can be done. The corresponding version of this result for weights concerns the manner in which pseudocones may be replaced by strict ones.

\begin{Theorem}\label{thm:char1}
A $T$-algebra $A$ is semiflexible, flexible, or pie, just when, respectively:
\begin{enumerate}[(a)]
\item There is a function assigning
    to each pseudomorphism $f \colon A \rightsquigarrow B$ a strict morphism $R(f)
    \colon A \to B$ which is isomorphic to $f$ in $\TAlg(A,B)$;\vskip0.5\baselineskip
\item Moreover, this $R$ may be chosen so that
    $R(f) = f$ when $f$ is strict, and so that $R(g \circ f) = g
    \circ R(f)$ for each $f \colon A \rightsquigarrow B$ and $g \colon B \to C$;\vskip0.5\baselineskip
\item Moreover, this $R$ may be chosen so that
    $R(g \circ R(f)) = R(g \circ f)$ for each $f \colon A \rightsquigarrow B$ and $g
    \colon B \rightsquigarrow C$.
\end{enumerate}
\end{Theorem}
In fact, we will only prove one direction of this implication here, namely that the semiflexibles, flexibles and pies satisfy clauses (a), (b) and (c) respectively. The other direction will be obtained through a cycle of implications which will be completed in the proofs of Theorems~\ref{thm:char2} and~\ref{thm:char3} below.
\begin{proof}
If $A$ is semiflexible, then $p_A \colon QA \to A$ admits a pseudoinverse $a \colon A \to QA$ in
$\TAlgs$. Now the inclusion $\TAlgs(A, \thg) \to \TAlg(A, \iota \thg)$ is equally the composite
\begin{equation*}
    \TAlgs(A, \thg) \xrightarrow{\TAlgs(p_A, \thg)} \TAlgs(QA, \thg)\xrightarrow{\cong}   \TAlg(A, \iota \thg) 
\end{equation*}
which consequently has pseudoinverse
\begin{equation*}
    R \colon \TAlg(A, \iota \thg) \xrightarrow{\cong} \TAlgs(QA, \thg) \xrightarrow{\TAlgs(a, \thg)} \TAlgs(A, \thg)
\end{equation*}
in $[\TAlgs, \Cat]$. This $R$
assigns to each pseudomorphism $f \colon A \rightsquigarrow B$ a strict morphism $R(f)$ which is isomorphic
to $f$, as required.

If now $A$ is flexible, then the $a$ above may be chosen to be a section of $p_A$, so that the corresponding $R$ satisfies $R(f) = f$ for any strict $f \colon A \to B$. That also $R(g \circ f) = g \circ R(f)$ for any $f \colon A \rightsquigarrow B$ and $g \colon B \to C$ is a direct consequence of $R$'s $2$-naturality.

Finally, if $A$ is pie, then the $a$ above may be chosen so as to exhibit $A$ as a $Q$-coalgebra; now given pseudo maps $f \colon A \rightsquigarrow B$ and $g \colon B \rightsquigarrow C$, corresponding to strict maps $\bar f \colon QA \to B$ and $\bar g \colon QB \to C$, we calculate that
\begin{align*}
R(g \circ R(f)) & = (\bar g \circ Q(\bar f \circ a)) \circ a = \bar g \circ Q\bar f \circ (Qa \circ a) \\
&= \bar g \circ Q\bar f \circ (\Delta_A \circ a) = (\bar g \circ Q \bar f \circ \Delta_A) \circ a = R(g \circ f)\rlap{ .}\qedhere
\end{align*}
\end{proof}
The version of this result for weights is now:
\begin{Theorem}\label{thm:char1weight}
A weight $W \in [\f J, \Cat]$ is semiflexible, flexible, or pie, just when, respectively:
\begin{enumerate}[(a)]
\item For every $2$-category $\f K$ and every $X \in \f K$, there is a function assigning to each $W$-weighted pseudocone $\alpha \colon X \dot{\rightsquigarrow} D$ a strict cone $R(\alpha) \colon X \dot{\rightarrow} D$ which is isomorphic to $\alpha$ in $\mathrm{PsCone}_W(X,D)$;\vskip0.5\baselineskip
\item Moreover, these $R$ may be chosen so that
    $R(\alpha) = \alpha$ whenever $\alpha$ is a strict cone, so that $R(\beta \circ \alpha) = \beta
    \circ R(\alpha)$ for each $\alpha \colon X \dot{\rightsquigarrow} D$ and $2$-natural $\beta \colon D \to D'$, and so that $R(\alpha \circ f)=R(\alpha)\circ f$ for each $\alpha \colon X \dot{\rightsquigarrow} D$ and $f \colon X^{\prime} \to X \in \f K$;\vskip0.5\baselineskip
\item Moreover, these $R$ may be chosen so that
    $R(\beta \circ R(\alpha)) = R(\beta \circ \alpha)$ for each $\alpha \colon X \dot{\rightsquigarrow} D$ and pseudonatural $\beta
    \colon D \rightsquigarrow D'$.
\end{enumerate}
\end{Theorem}
Again, we prove only the forward implication, with the other direction now following from the proofs of Theorems~\ref{thm:char2weight} and~\ref{thm:char3weight} below.
\begin{proof}
Viewing $W$ as an algebra for the weight 2-monad on $[\mathrm{ob}\ \f J,\Cat]$, the strict and pseudo cones from $X$ to $D$ are now strict and pseudo algebra morphisms with domain $W$.  This being the case, we see that each clause of the present theorem has a corresponding clause in Theorem~\ref{thm:char1} of which it is an immediate consequence; the only exception being that in (b) above we also demand that $R(\alpha \circ f) = R(\alpha) \circ f$ for each $\alpha \colon X \dot{\rightsquigarrow} D$ and each $f \colon X' \to X$ in $\f K$.  However, the pseudocone $\alpha \circ f$ is given by the composite
\begin{equation*}
\cd[@C+1em]{
W \ar@{~>}[r]^-\alpha & \f K(X, D \thg) \ar[r]^{\f K(f, D\thg)} & \f K(X', D\thg)\rlap{ ,}
}
\end{equation*}
and since $\f K(f,D\thg)$ is 2-natural and occurs to the right of $\alpha$, we see that this clause too follows from Theorem~\ref{thm:char1}(b).
\end{proof}
Recall that, given $W \in [\f J, \Cat]$ and $D \colon \f J \to \f K$, a \emph{bilimit of $D$ weighted by $W$} is a birepresentation for the $2$-functor $\mathrm{PsCone}_W(\thg, D)$; that is, a pseudocone $\alpha \colon U \dot{\rightsquigarrow} D$, composition with which induces an equivalence of categories $\f K(X, U) \to \mathrm{PsCone}_W(X, D)$ for each $X \in \f K$. Now given a limit $\{W, D\}$ existing in a 2-category $\f K$, we may ask whether its limiting cone $\alpha \colon \{W, D\} \dot{\to} D$ also exhibits $\{W, D\}$ as the $W$-weighted bilimit of $D$; which is equally to ask that the composite
\begin{equation}\label{eq:coneps}
\f K(X,\{W, D\}) \cong \mathrm{Cone}_W(X, D) \to \mathrm{PsCone}_W(X, D)
\end{equation}
should be an equivalence for each $X \in \f K$. If this is so for every $W$-weighted limit existing in every $2$-category, let us then say that \emph{limits weighted by $W$ are bilimits}.
\begin{Corollary}\label{cor:semiflexbi}
 A weight $W$ is semiflexible if and only if limits weighted by $W$ are bilimits.
\end{Corollary}
\begin{proof}
If $W \in [\f J, \Cat]$ is a semiflexible weight, then by Theorem~\ref{thm:char1weight}(a), each inclusion $\mathrm{Cone}_W(X, D) \to \mathrm{PsCone}_W(X, D)$ is essentially surjective on objects;  since these inclusions are always fully faithful, we conclude that they are equivalences. Thus for every suitable $\f K$, $D$ and $X$, the functor~\eqref{eq:coneps} is the composite of two equivalences and so itself an equivalence.

For the converse, first observe that if both the limit $\{W,F\}$ and pseudolimit $\{W,F\}_{\mathrm{ps}}$ of a diagram $F$ exist in some $2$-category then the limit $\{W,F\}$ is a bilimit just when the canonical map $\{W,F\} \to \{W,F\}_{\mathrm{ps}}$ is an equivalence.  Now letting $X$, $D$ and $\f K$ be arbitrary as before, both the limit and pseudolimit in \Cat of the diagram $\f K(X,D\thg) \colon  \f J \to \Cat$ do exist and the canonical map $\{W,\f K(X,D \thg)\} \to \{W,\f K(X,D \thg)\}_{\mathrm{ps}}$ is precisely the inclusion $\mathrm{Cone}_W(X, D) \to \mathrm{PsCone}_W(X, D)$.  Thus if limits weighted by $W$ are bilimits, this map is an equivalence for any choice of $X$, $D$ and $\f K$; in particular, it is essentially surjective on objects so that by Theorem~\ref{thm:char1weight}(a), $W$ is semiflexible. 
\end{proof}

\subsection{Second characterisation}

We now give our second characterisation result; as for our other results, it will have six versions, one for each of the semiflexible, flexible and pie algebras, and one for each corresponding class of weights. It takes its most intuitive form in the case of the pie weights: here it states that a weight $W$ is pie just when for every complete category $\f K$, the weighted limit $2$-functor $\{W, \thg\} \colon [\f J, \f K] \to \f K$ may be extended to a $2$-functor $\Ps(\f J, \f K) \to \f K$. That such extensions exist for pie weights was shown as a special case of~\cite[Proposition~5.8]{Lack2011Enhanced}; what was not shown there is that such extensions exist \emph{only} for pie weights. The corresponding result for a general pie algebra $A$ is that the hom $2$-functor $\TAlgs(A, \thg) \colon \TAlgs \to \Cat$ may be extended to a $2$-functor $\TAlg \to \Cat$, whilst in the flexible and semiflexible situations, we obtain similar extensions, but of successively weaker kinds. 

Let us now give these results, first for algebras and then for weights:

\begin{Theorem}\label{thm:char2}
A $T$-algebra $A$ is semiflexible, flexible, or pie, just when, respectively:
\begin{enumerate}[(a)]
\item The hom $2$-functor $\TAlgs(A, \thg)$ admits an extension
\begin{equation}\label{eq:rextension}
    \cd[@C+1em]{
        \TAlgs \ar[rr]^{\iota} \ar[dr]_{\TAlgs(A, \thg)\ \ \ } & \rtwocell{d}{\theta} & \TAlg
        \ar@{~>}[dl]^{\ \ \ H} \\ &
        \Cat    }
\end{equation}
where $H$ is a pseudofunctor and $\theta$ an invertible icon; \vskip0.5\baselineskip
\item This extension may moreover be chosen so that $\theta$ is an identity, and so that $H$'s pseudofunctoriality constraints $Hg \circ Hf \to H(g \circ f)$ are identities whenever either $f$ or $g$ is strict;\vskip0.5\baselineskip
 \item This extension may moreover be chosen so that $H$ is a $2$-functor.
\end{enumerate}
\end{Theorem}
The notion of \emph{icon} in clause (a) of the statement of this theorem is as it was in Section~\ref{sec:2catsrelated}, though now we are concerned only with \emph{invertible} icons: these are pseudonatural transformations all of whose $1$-cell components are all identities.
\begin{proof}
We continue our cycle of implications; thus we must show that clauses (a), (b) and (c) of Theorem~\ref{thm:char1} imply the corresponding clauses of the present theorem.  In each case this will be achieved by establishing, successively stronger, properties of the 2-natural transformation
\begin{equation}\label{eq:iotainduced}
\cd[@C+1em]{
  \TAlgs \ar[rr]^\iota \ar[dr]_{\TAlgs(A, \thg)\ \ \ } & \rtwocell[0.4]{d}{} & \TAlg \ar[dl]^{\ \ \ \TAlg(A, \thg)} \\
&  \Cat
}
\end{equation}
induced by $\iota$.  These properties of ~\eqref{eq:iotainduced} will feed into the corresponding clause of Lemma~\ref{lem:importantlemma} below, which is concerned with the construction of extensions, and the theorem will follow directly.

Assume first that Theorem~\ref{thm:char1}(a) holds; then the function $R$ given there witnesses~\eqref{eq:iotainduced} as pointwise essentially surjective on objects. Since~\eqref{eq:iotainduced} is always pointwise fully faithful, it is therefore in this case a pointwise equivalence.  Clause (a) of the present theorem now follows directly upon application of Lemma~\ref{lem:importantlemma}(a).

In the situation of Theorem~\ref{thm:char1}(b) the function $R$ is also assumed to satisfy $Rf=f$ whenever $f$ is strict; such an $R$ determines a family of pointwise sections $\rho_{B} \colon \TAlg(A,B) \to \TAlgs(A,B)$ of~\eqref{eq:iotainduced}. The further condition imposed on $R$ in Theorem~\ref{thm:char1}(b) amounts to the assertion that the family $\rho$ is 2-natural; thus a section of~\eqref{eq:iotainduced}, which is consequently an injective equivalence in $[\TAlgs, \Cat]$.  Clause (b) of the present theorem now follows directly from Lemma~\ref{lem:importantlemma}(b).

Clause (c) of Theorem~\ref{thm:char1} imposes one further condition on $R$, which, rephrased in terms of $\rho$, asserts precisely that the diagram:
\begin{equation*}
\cd[@C+0.5em]{
	\TAlg(A,B) \ar[d]_{\TAlg(A,g)} \ar[r]^-{\rho_{B}} &
	\TAlgs(A,B) \ar[r]^-{\iota} &
	\TAlg(A,B) \ar[d]^{\TAlg(A,g)} \\
    \TAlg(A,C) \ar[r]_-{\rho_{C}} &
	\TAlgs(A,C) &
	\TAlg(A,C) \ar[l]^-{\rho_{C}}
}
\end{equation*}
is commutative for all $g \colon B \rightsquigarrow C$ in $\TAlg$; and now clause (c) of the present theorem follows directly from Lemma~\ref{lem:importantlemma}(c).
\end{proof}
\begin{Theorem}\label{thm:char2weight}
A weight $W \in [\f J, \Cat]$ is semiflexible, flexible, or pie, just when, respectively:
\begin{enumerate}[(a)]
\item For all complete $2$-categories $\f K$, the limit $2$-functor $\{W, \thg\}\colon [\f J, \f K] \to \f K$ admits an extension
\begin{equation}\label{eq:rextensionweight}
    \cd[@C+1em]{
        [\f J, \f K] \ar[rr]^-{\iota} \ar[dr]_{\{W, \thg\}\ \ } & \rtwocell{d}{\theta} &
        \Ps(\f J, \f K) \ar@{~>}[dl]^{\ \ H}\\
        & \f K
    }
\end{equation}
where $H$ is a pseudofunctor and $\theta$ an invertible icon; \vskip0.5\baselineskip
\item This extension may moreover be chosen so that $\theta$ is an identity, and so that $H$'s pseudofunctoriality constraints $Hg \circ Hf \to H(g \circ f)$ are identities whenever either $f$ or $g$ is strict;\vskip0.5\baselineskip
 \item This extension may moreover be chosen so that $H$ is a $2$-functor.
\end{enumerate}
\end{Theorem}
\begin{proof}
We continue our cycle of implications by showing that clauses (a), (b) and (c) of Theorem~\ref{thm:char1weight} imply the corresponding clauses of the present theorem.  This time, we do so by establishing, for a given complete $2$-category $\f K$, successively stronger properties of the 2-natural transformation
\begin{equation*}
\cd[@C+1em]{
  [\f J, \f K] \ar[rr]^-\iota \ar[dr]_{\{W, \thg\}\ \ } & \rtwocell[0.4]{d}{\gamma} &
  \Ps(\f J, \f K) \ar[dl]^{\ \ \{W, \thg\}_\mathrm{ps}} \\ &
  \f K
}
\end{equation*}
induced by the canonical comparison maps $\{W, D\} \to \{W, D\}_\mathrm{ps}$; these properties fed into the corresponding clauses of Lemma~\ref{lem:importantlemma} then imply the present theorem.

Let us first show, assuming Theorem~\ref{thm:char1weight}(a), that $\gamma$ is a pointwise equivalence.  This will be the case precisely if $\f K(X, \gamma)$ is a pointwise equivalence for each $X \in \f K$.  But $\f K(X, \gamma)$ is isomorphic to the inclusion $\mathrm{Cone}_W(X,\thg) \to \mathrm{PsCone}_W(X,\thg)$, which is always pointwise fully faithful, and is pointwise essentially surjective as witnessed by the function $R$ of Theorem~\ref{thm:char1weight}(a).  Now applying Lemma~\ref{lem:importantlemma}(a) yields clause (a) of the present theorem.

Suppose next that Theorem~\ref{thm:char1weight}(b) holds. To give a section of $\gamma$ in $[[\f J, \f K], \f K]$ is, by the Yoneda lemma, to give sections of each $\f K(X, \gamma)$ which are $2$-natural in $X \in \f K$; which is equally well to give sections of the isomorphic
$$\mathrm{Cone}_W(X,\thg) \to \mathrm{PsCone}_W(X,\thg)\rlap{ ,}$$
$2$-naturally in $X$. The witnessing function $R$ of Theorem~\ref{thm:char1weight}(b) provides just such a section; assuming this we therefore obtain a section $\rho$ of $\gamma$, which is consequently an injective equivalence in $[[\f J, \f K], \f K]$.  Clause (b) of the theorem now follows directly from Lemma~\ref{lem:importantlemma}(b).

The additional property of $R$ specified in Theorem~\ref{thm:char1weight}(c) now translates, in terms of $\rho$, to the commutativity of
\begin{equation*}
\cd[@C+0.5em]{
	\{W, A\}_\mathrm{ps} \ar[d]_{\{W, g\}_\mathrm{ps}} \ar[r]^{\rho_{A}} &
	\{W, A\} \ar[r]^{\gamma_{A}} &
	\{W, A\}_\mathrm{ps} \ar[d]^{\{W, g\}_\mathrm{ps}} \\
    \{W, B\}_\mathrm{ps} \ar[r]_{\rho_{B}} &
	\{W, B\} &
	\{W, B\}_\mathrm{ps} \ar[l]^{\rho_{B}}
}
\end{equation*}
for all pseudonatural transformations of diagrams $g \colon A \rightsquigarrow B$; and clause (c) of the present theorem now follows directly from Lemma~\ref{lem:importantlemma}(c).
\end{proof}
We now give the lemma which was used in the proof of the above two results.
\begin{Lemma}\label{lem:importantlemma}
Consider a diagram \begin{equation*}
    \cd[@+1em]{
        \f A \ar[rr]^{I} \ar[dr]_{F} & \rtwocell{d}{\alpha} & \f B
        \ar[dl]^{ G} \\ &
        \f C    }
\end{equation*}
of $2$-categories, $2$-functors and a pseudonatural transformation $\alpha$.
\begin{enumerate}[(a)]
\item
If $I$ is bijective on objects and $\alpha$ a pointwise equivalence, then this diagram admits a factorisation as
\begin{equation*}
    \cd[@+1em]{
        \f A \ar[rr]^{I} \ar[dr]_{F} & \rtwocell{d}{\alpha_1} \rtwocell[0.7]{dr}{\alpha_2}& \f B
        \ar@{~>}[dl]|{H} \ar@/^26pt/[dl]^{G}\\ &
        \f C   & {} }
\end{equation*}
where $H$ is a pseudofunctor, $\alpha_1$ an invertible icon and $\alpha_2$ a pseudonatural equivalence. \vskip0.5\baselineskip
\item If in the situation of (a), $\alpha$ is in fact $2$-natural, and moreover admits a retraction $\beta$ in $[\f A, \f C]$, then the $\alpha_1$ of this factorisation may be chosen to be the identity;
it then follows that $H$ strictly preserves composition with maps from $\f A$, in the sense that the pseudofunctoriality constraint $Hg \circ Hf \to H(g \circ f)$ is an identity whenever either $f$ or $g$ is in the image of $I$.
\vskip0.5\baselineskip
\item If in the situation of (b), the retraction $\beta$
may be chosen so that
the diagram
\begin{equation}\label{eq:tobeatwofunctordiag}
\cd[@C+0.5em]{
	GIB \ar[d]_{Gg} \ar[r]^-{\beta_B} &
	FB \ar[r]^-{\alpha_B} &
	GIB \ar[d]^{Gg} \\
	GIC \ar[r]_-{\beta_C} &
	FC &
    GIC \ar[l]^{\beta_C}
}
\end{equation}
commutes for every $g \colon IB \to IC$ in $\f B$, then $H$ may be taken to be a $2$-functor.
\end{enumerate}
\end{Lemma}
\begin{proof}
We start with (a). Without loss of generality we may assume that $I$ is in fact the identity on objects. Let $\mathrm{Hom}(\f B, \f C)$ denote the $2$-category of pseudofunctors, pseudonatural transformations and modifications from $\f B$ to $\f C$.  We have the forgetful $2$-functor $U \colon \mathrm{Hom}(\f B, \f C) \to [\mathrm{ob}\ \f B, \f C]$ and it is well-known that this admits the lifting of adjoint equivalences. Now for
each $X \in \mathrm{ob}\ \f A = \mathrm{ob}\ \f B$ we have the map $\alpha_X \colon FX \to GIX = GX$ which by assumption is an equivalence; so choosing an adjoint pseudoinverse $\beta_X$ for each $\alpha_X$, we obtain an adjoint equivalence $(\alpha, r) \colon \mathrm{ob}\ F \leftrightarrows  UG \colon (\beta, s)$ in $[\mathrm{ob}\ \f B, \f C]$. Lifting this along $U$ we obtain a pseudofunctor $H \colon \f B \rightsquigarrow \f C$ and pseudonatural equivalence $\alpha_2 \colon H \rightsquigarrow G$ with $U\alpha_2 = \alpha$. Explicitly, $H$ has the same action on objects as $F$, and on morphisms sends $g \colon A \to B$ in $\f B$ to the composite
\begin{equation*}
  FA \xrightarrow{\alpha_A} GA \xrightarrow{Gg} GB \xrightarrow{\beta_B} FB
\end{equation*}
obtained by conjugating $Gg$ by the adjoint equivalences. The action on $2$-cells is similar, whilst the unit and composition coherence constraints are obtained from the unit and counit maps $r_A$ and $s_A$ in the evident manner. As for the pseudonatural transformation $\alpha_2$, this has its $1$-cell components being those of $\alpha$ and its $2$-cell component at $g\colon A \to B$ in $\f B$ being given by the pasting composite
\begin{equation*}
\cd[@+1em]{
  FA \ar[r]^{\alpha_A} \ar[d]_{\alpha_A} & GA \ar[r]^{Gg} & GB \ar[r]^{\beta_B} \ar[dr]_1 & FB \dtwocell[0.35]{dl}{s_B}
  \ar[d]^{\alpha_B} \\
  GA \ar[rrr]_{Gg} & & & GB\rlap{ .}
}
\end{equation*}
This completes our description of $H$ and $\alpha_2$, and we now turn to $\alpha_1$. As indicated above, its $1$-cell components are identities; whilst at a morphism $f \colon A \to B$ of $\f A$ its $2$-cell component is given by the top row of the following pasting diagram:
\begin{equation*}
\cd[@+1em]{
  FA \ar@{=}[d] \ar[rr]^{Ff} \dtwocell{drr}{\alpha_f} & & FB \ar[d]_{\alpha_B} \ar[r]^{1} \dtwocell{dr}{r_B} & FB \ar@{=}[d] \\
  FA \ar[r]_{\alpha_A} \ar[d]_{\alpha_A} & GA \ar[r]_{GIf} & GB \ar[r]|{\beta_B} \ar[dr]_1 & FB \dtwocell[0.35]{dl}{s_B}
  \ar[d]^{\alpha_B} \\
  GA \ar[rrr]_{GIf} & & & GB\rlap{ .}
}
\end{equation*}
The bottom row of this diagram is $(\alpha_2I)_f$ and so the entire pasting constitutes the value of $(\alpha_2I \circ \alpha_1)_f$; which, since $r_B$ and $s_B$ cancel by the triangle identities, is equal to $\alpha_f$, so showing that $\alpha_2 I \circ \alpha_1 = \alpha$ as required. Note that we have not checked the coherence axioms for $\alpha_1$ but since its composite with the pseudonatural $\alpha_2 I$ is pseudonatural, the same is true of $\alpha_1$: \emph{pseudonatural equivalences detect pseudonaturality}.

This completes the proof of (a); as for (b), suppose now that $\alpha$ is in fact $2$-natural, and that the pseudoinverses $\beta_X$ chosen for each $\alpha_X$ above constitute a retraction for $\alpha$ in $[\f A, \f C]$. Then each $2$-cell $\alpha_f$ is the identity, as is each $r_B \colon 1 \to \beta_B \alpha_B$, so that each $2$-cell component of $\alpha_1$, being a composite of such $2$-cells, is itself the identity. Now given maps $f \colon A \to B$ in $\f B$ and $g \colon B \to C$ in $\f A$, we must show that the pseudofunctoriality constraint $HIg \circ Hf \to H(Ig \circ f)$ is the identity, or in other words, that the whiskering
\begin{equation*}
\cd{
  FA \ar[r]^{\alpha_A} & GA \ar[r]^{Gf} & GB \ar@/_2.5em/[rr]_{1} \ar[r]^{\beta_B} &
  FB \ar[r]^{\alpha_B} & GB \ar[r]^{GIg} & GC \ar[r]^{\beta_C} & FC \\
  & & & {} \dtwocell[0.55]{u}{s_B}
}
\end{equation*}
is the identity: which is so as $\beta_C \circ GIg \circ s_B = Fg \circ \beta_b \circ s_B = 1_{Fg}$. The argument for the dual case, where $f \colon A \to B$ in $\f A$ and $g \colon B \to C$ in $\f B$, is similar.

Finally, we prove (c). Given a retraction $\beta$ as before, let us suppose that its components make each diagram~\eqref{eq:tobeatwofunctordiag} commute; we will show that $H$ is then a $2$-functor. Since $F$ is a $2$-functor and $HI = F$, it already follows that $H$ preserves identities strictly. As for binary composition, we must show that for every $f \colon A \to B$ and $g \colon B \to C$ in $\f B$ the whiskering
\begin{equation*}
\cd{
  FA \ar[r]^{\alpha_A} & GA \ar[r]^{Gf} & GB \ar@/_2.5em/[rr]_{1} \ar[r]^{\beta_B} &
  FB \ar[r]^{\alpha_B} & GB \ar[r]^{Gg} & GC \ar[r]^{\beta_C} & FC \\
  & & & {} \dtwocell[0.55]{u}{s_B}
}
\end{equation*}
is an identity $2$-cell; we will show in fact that the composite 2-cell $\beta_C \circ Gg \circ s_B$ on the right of the whiskering is an identity.  Commutativity in~\eqref{eq:tobeatwofunctordiag} asserts that $\beta_C \circ Gg \circ s_B$ is an endomorphism; to show that it is the identity, we observe that $\beta_C \circ Gg \circ s_B \circ \alpha_B = 1_{\beta_C \circ Gg}  \circ \alpha_B$ by the triangle identities, and conclude that $\beta_C \circ Gg \circ s_B = 1_{\beta_C \circ Gg}$ since $\alpha_B$ is an equivalence.
\end{proof}

\subsection{Third characterisation}
We now turn to our final characterisation result, which is most easily motivated by the following question concerning limits: for which weights $W \in [\f J, \Cat]$ does the  weighted limit $2$-functor $\{W, \thg\} \colon [\f J, \f K] \to \f K$ send pointwise equivalences to equivalences  for every complete category $\f K$? 
A closely related question was considered by Par\'e~\cite{PareDouble}; his notion of \emph{persistent limit} concerns weights which display this good behaviour, but with respect to a double-categorical, and not $2$-categorical, notion of pointwise diagram equivalence. It was shown by Verity in~\cite{Verity1992Enriched} that
the persistent weights are precisely the flexible ones. 
However, for our purely $2$-categorical question, the class of weights answering to it turns out to be the larger class of semiflexibles; amongst which the flexibles are characterised by their also sending pointwise \emph{surjective} equivalences to surjective equivalences.

As is well-known, pointwise equivalences or surjective equivalences in a functor category $[\f J, \f K]$ need not be genuine ones; were this the case, then every weight would answer to the above characterisations, since $\{W, \thg\}$, like any $2$-functor,  preserves such equivalences. Yet the inclusion $\iota \colon [\f J, \f K] \to \Ps(\f J, \f K)$ sends pointwise equivalences or surjective equivalences to genuine ones: and so by asking that the $2$-functor $\{W, \thg\} \colon [\f J, \f K] \to \f K$ admit a suitable extension to $\Ps(\f J, \f K)$, as in Theorem~\ref{thm:char2weight}, it will follow that $\{W, \thg\}$ sends pointwise equivalences or pointwise surjective equivalences, as appropriate, to genuine ones; this is the core of our characterisation.

In order to give the corresponding results for semiflexible and flexible algebras, we need a notion corresponding to that of pointwise diagram equivalence. In the case that $\f K$ is cocomplete, we may view $[\f J, \f K]$ as the algebras for a $2$-monad on $\f C = [\mathrm{ob}\ \f J, \f K]$, and now the pointwise equivalences therein are equally well the algebra maps which become equivalences under the forgetful $2$-functor $U \colon \TAlgs \to \f C$. We may consider this same class of algebra maps for a general $2$-monad, calling them \emph{$U$-equivalences}: and the basic form of our result is now that an algebra $A$ is semiflexible just when $\TAlgs(A, \thg)$ sends $U$-equivalences to genuine ones.
Similarly, we have the notion of \emph{$U$-surjective equivalence}---algebra maps which become surjective equivalences on applying $U$---and with respect to these, the flexibles have a corresponding characterisation. 

Given a $2$-monad $T$ with rank on a locally presentable $2$-category $\f C$, it is shown in~\cite{Lack2007Homotopy-theoretic} that the $U$-equivalences and the $U$-surjective equivalences are the weak equivalences and trivial fibrations of a naturally-arising Quillen model structure on~\TAlgs. The corresponding cofibrant objects are the flexible algebras, and our characterisation of them as those objects for which $\TAlgs(A, \thg)$ sends $U$-surjective equivalences to surjective equivalences is now contained in Proposition~2.3 of that paper. 

We have so far not discussed the pie case; this can be understood in terms of the \emph{algebraic model structures} of~\cite{Riehl2011Algebraic}. The notion of algebraic model structure strengthens the classical one in a number of respects; one of these is that cofibrant replacement becomes a comonad. The model structure on $\TAlgs$ described in the previous paragraph can be made algebraic in such a way that the comonad in question is the $Q$ of our considerations: and now the coalgebras for this comonad, our pie algebras, are the \emph{algebraically cofibrant} objects of this model structure. These may be characterised by their bearing a coherent choice of liftings against the \emph{algebraic trivial fibrations}; and for the algebraic model structure on $\TAlgs$, an algebraic trivial fibration is composed of a $U$-surjective equivalence $f$ together with a chosen section of $Uf$---thus a witness to $f$'s being a $U$-surjective equivalence. It is to such maps that the pie case of our characterisation theorem for algebras will refer. 

\begin{Theorem}\label{thm:char3}
A $T$-algebra $A$ is semiflexible, flexible, or pie, just when, respectively:
\begin{enumerate}[(a)]
\item The hom $2$-functor $\TAlgs(A, \thg)$ sends $\us$-equivalences to equivalences;
\item The hom $2$-functor $\TAlgs(A, \thg)$ sends $\us$-surjective equivalences to surjective equivalences;
\item The hom $2$-functor $\TAlgs(A, \thg)$ coherently transports $\us$-split surjective equivalences to split surjective equivalences.
\end{enumerate}
\end{Theorem}
In clauses (a) and (b), the $U$-equivalences and $U$-surjective equivalences are as before those algebra maps which $U$ sends to equivalences or surjective equivalences; but these are equally well, as in Proposition~4.10 of~\cite{Lack2007Homotopy-theoretic}, those maps whose image under $\iota \colon \TAlgs \to \TAlg$ is an equivalence or surjective equivalence. As for clause (c) we say that a 2-functor $F \colon \TAlgs \to \f C$ coherently transports $\us$-split surjective equivalences to split surjective equivalences if it sends $U$-surjective equivalences to surjective equivalences and if, furthermore, there exists a function $\varphi$ which to each $\us$-surjective equivalence $f \colon B \to C$ in $\TAlgs$ and each splitting $k \colon \us B \to \us A$ for $\us f$ assigns a splitting $\varphi_f(k) \colon FC \to FB$ for $Ff$, all subject to two axioms: firstly, that $\varphi_f(\us k) =Fk$ whenever $k$ is a splitting for $f$ in $\TAlgs$, and secondly, that $\varphi_{gf}(k\ell) = \varphi_f(k) \circ \varphi_g(\ell)$ whenever this makes sense.
\begin{proof}
We will show that clauses (a), (b) and (c) of Theorem~\ref{thm:char2} imply the corresponding clauses of the current result; then we show that these clauses in turn imply that an algebra is respectively semiflexible, flexible or pie. This will complete our cycle of implications.

As remarked above, if $f \colon B \to C$ is a $\us$-equivalence in $\TAlgs$ then $\iota(f)$ is an equivalence in $\TAlg$.  If $f$ is a $\us$-surjective equivalence with splitting $k \colon \us C \to \us B$ in $\f C$, fully faithfulness of $\us f$ ensures that there exists a unique pseudomorphism structure on this $k$ making it into a section $\bar k \colon C \rightsquigarrow B$ for $\iota f$.  It follows from the uniqueness of this correspondence that the function $\varphi_f(k)= \bar k$ exhibits the inclusion $\iota \colon \TAlgs \to \TAlg$ as coherently transporting $\us$-split surjective equivalences to split surjective equivalences.

Now suppose that Theorem~\ref{thm:char2}(a) holds, so that $\TAlgs(A, \thg)$ admits an extension $(H, \theta)$ as in~\eqref{eq:rextension}. If $f \colon B \to C$ is a $\us$-equivalence in $\TAlgs$, then $\iota(f)$ is an equivalence in $\TAlg$, and hence $H\iota(f)$ is one in $\Cat$. But $\TAlgs(A, f) \cong H\iota(f)$ via $\theta_f$ and so $\TAlgs(A, f)$ is an equivalence in $\Cat$ also. This verifies clause (a) of the current result.

Next suppose that Theorem~\ref{thm:char2}(b) holds; we must show that $\TAlgs(A, \thg)$ sends each $\us$-surjective equivalence $f \colon B \to C$ to a surjective equivalence. Certainly $\TAlgs(A, f)$ is an equivalence by the case just proved, and it remains to show that it admits a section. Let $\us f$ admit a section $k \colon \us C \to \us B$ with $\bar k \colon C \rightsquigarrow B$ as above the corresponding section for $\iota f$. We now calculate that $H\iota f \circ H\bar k = H(\iota f \circ \bar k) = H(1_C) = H\iota(1_C) = \TAlgs(A, 1_C) = 1$, so that $H\bar k$ is a section for $H\iota f = \TAlgs(A, f)$, as required.

Finally, suppose that Theorem~\ref{thm:char2}(c) holds; we must then show that $\TAlgs(A, \thg)$ coherently transports $\us$-split surjective equivalences to split surjective equivalences.  The inclusion $\iota \colon \TAlgs \to \TAlg$ does so, with associated function $\varphi_f(k)= \bar k$, and it follows, as with any 2-functor based on $\TAlg$, that $H\iota=\TAlgs(A, \thg)$ has the same transport property, now with associated function $\varphi_f(k)= H \bar k$.

It remains to show that clauses (a), (b) and (c) of the current result imply that an algebra $A$ is semiflexible, flexible or pie respectively. So suppose first that $\TAlgs(A, \thg) \colon \TAlgs \to \Cat$ sends $\us$-equivalences to equivalences. Observe that $p_A \colon QA \to A$ is a $\us$-equivalence, since $U\qun_A \colon UA \to UQA$ is a pseudoinverse for it in $\f C$; and hence $\TAlgs(A, p_A) \colon \TAlgs(A, QA) \to \TAlgs(A,A)$ is an equivalence in $\Cat$. In particular its essential surjectivity means that we can find $a \in \TAlgs(A, QA)$ with $p_A \circ a \cong 1_A$; whence $A$ is semiflexible.

Suppose now that $\TAlgs(A, \thg)$ sends $\us$-surjective equivalences to surjective equivalences. The map $p_A \colon QA \to A$ is in fact a $\us$-surjective equivalence, since $Uq_A$ is a splitting for $Up_A$, and so $\TAlgs(A, p_A)$ is in this case a surjective equivalence in $\Cat$. Now surjectivity allows us to find $a \colon A \to QA$ with $p_A \circ a = 1_A$, so that $A$ is flexible.

Finally, suppose that $\TAlgs(A, \thg)$ comes equipped with a function
$\varphi$ which coherently transports $\us$-split surjective equivalences in
$\TAlgs$ to split surjective equivalences in $\Cat$. Then to the splitting
$\us\qun_A \colon \us A \to \us QA$ of $\us p_A$ we have assigned a splitting
$\varphi(\us\qun_A) \colon \TAlgs(A, A) \to \TAlgs(A, QA)$ of $\TAlgs(A, p_A)$, and we claim that the morphism $a =
\varphi(\us\qun_A)(1_A) \colon A \to QA$ exhibits $A$ as a strict $Q$-coalgebra.
Clearly we have $p_A \circ a = 1_A$ and so it remains to prove coassociativity.
To obtain this, consider the following two equalities of $\us$-split surjective
equivalences in $\TAlgs$:
\begin{align*}
  \cd{
  QQA \ar[r]_{p_{QA}} \ar@{<~}@/^1.8em/[r]^{\us\qun_{QA}} & QA \ar[r]_{p_A} \ar@{<~}@/^1.8em/[r]^{\us\qun_A} & A
  } \quad &= \quad
\cd{
  QQA \ar[r]_{Qp_A} \ar@{<-}@/^1.8em/[r]^{\us Q\qun_{A}} & QA \ar[r]_{p_A} \ar@{<~}@/^1.8em/[r]^{\us\qun_A} & A}
  \\
 \text{and} \qquad \cd{
  QQA \ar[r]_{p_{QA}} \ar@{<~}@/^1.8em/[r]^{\us\qun_{QA}} & QA \ar[r]_{p_A} \ar@{<-}@/^1.8em/[r]^{\us a} & A
  } \quad &= \quad
\cd{
  QQA \ar[r]_{Qp_{A}} \ar@{<-}@/^1.8em/[r]^{\us Qa} & QA \ar[r]_{p_A} \ar@{<~}@/^1.8em/[r]^{\us\qun_A} & A
  }
\end{align*}
where the left pointing arrows are the splittings in \f{C}, and drawn strictly if they underlie a morphism of \TAlgs.  Applying the transport function $\varphi$ to each of these, and using its coherence, we obtain from the first equality the commutative square below on the left, and from the second, that on the right:
\begin{equation*}
\cd[@C+2em]{
  \TAlgs(A, A) \ar[d]_{\varphi(\us\qun_A)} \ar[r]^-{\varphi(\us\qun_A)} &
  \TAlgs(A, QA) \ar[d]^{\varphi(\us\qun_{QA})} &
  \TAlgs(A, A) \ar[d]^{\varphi(\us\qun_A)} \ar[l]_-{a \circ (\thg)} \\
  \TAlgs(A, QA) \ar[r]_-{Q\qun_A \circ (\thg)} &
  \TAlgs(A, QQA) &
  \TAlgs(A, QA) \ar[l]^-{Qa\circ (\thg)}
}
\end{equation*}
Now evaluating both commutative squares at $1_A \in \TAlgs(A,A)$ we obtain from the left one
the equality $Q\qun_A \circ a = \varphi(\us\qun_{QA})(a)$ and from the right one the equality
$Qa \circ a = \varphi(\us\qun_{QA})(a)$; whence $Qa \circ a = Q\qun_A \circ a = \Delta_A \circ a$ and so $a \colon A \to QA$ is a strict coalgebra as required.
\end{proof}

Finally, we give the corresponding version of the above theorem for the case of weights. In the third clause, a \emph{pointwise split surjective equivalence} is a pointwise surjective equivalence equipped with a chosen section of each component. We note that if $\f K$ is cocomplete, so that $[\f J, \f K]$ is $\TAlgs$ for a $2$-monad $T$ on $[\mathrm{ob}\ \f J, \f K]$, then such maps are precisely the $U$-split surjective equivalences for this $T$. The notion of coherent transport of pointwise split surjective equivalences may be transcribed directly from the corresponding transport notion for $U$-split surjective equivalences, and is again a special case of it when $\f K$ is cocomplete.
\begin{Theorem}\label{thm:char3weight}
A weight $W \in [\f J, \Cat]$ is semiflexible, flexible, or pie, just when, respectively:
\begin{enumerate}[(a)]
\item For all complete $2$-categories $\f K$, the $2$-functor $\{W, \thg\}  \colon [\f J, \f K] \to \f K$ sends pointwise equivalences to equivalences;
\item For all complete $2$-categories $\f K$, the $2$-functor $\{W, \thg\}  \colon [\f J, \f K] \to \f K$ sends pointwise surjective equivalences to surjective equivalences;
\item For all complete $2$-categories $\f K$, the $2$-functor $\{W, \thg\}  \colon [\f J, \f K] \to \f K$ coherently transports pointwise split surjective equivalences to split surjective equivalences.
\end{enumerate}
\end{Theorem}
\begin{proof}
The argument that Theorem~\ref{thm:char2weight}'s clauses (a), (b) and (c) imply those of the present result is exactly as in the preceding proof. That these in turn imply that a weight $W \in [\f J, \Cat]$ is semiflexible, flexible or pie is now a direct consequence of Theorem~\ref{thm:char3}. Indeed, if we view $[\f J, \Cat]$ as $\TAlgs$ for the weight $2$-monad $T$ on $[\mathrm{ob}\ \f J, \Cat]$, then the pointwise equivalences and their two variants become for this $T$ the $U$-equivalences and their corresponding variants. Moreover, there is a natural isomorphism $\{W, \thg\} \cong [\f J, \Cat](W, \thg) \colon [\f J, \Cat] \to \Cat$, and so clauses (a), (b) and (c) of the present theorem may be identified for this $T$ with the corresponding clauses of Theorem~\ref{thm:char3}. Thus assuming each to hold in turn, we apply that result to deduce that $W$ is, respectively, a semiflexible, flexible or pie algebra; that is, a semiflexible, flexible or pie weight.
\end{proof}

Let us conclude this section by remarking upon some further properties of flexible algebras and weights which do not quite fit into our series of characterisations.  In Theorem~\ref{thm:char3}(b) we characterised the flexibles as those algebras $A$ for which the hom 2-functor $\TAlgs(A, \thg)$ sends $\us$-surjective equivalences to surjective equivalences.  The argument by which we proved a flexible algebra $A$ to have this property required only that $\TAlgs(A, \thg)$ admit an extension $H \colon \TAlgs \to \Cat$ preserving identities strictly and satisfying $H(g \circ f)=H(g) \circ H(f)$ whenever $g$ is strict.  But in fact Theorem~\ref{thm:char2}(b) also asserts that $H(g \circ f)=H(g) \circ H(f)$ whenever $f$ is strict; and thus by a corresponding argument we conclude that \emph{if $A$ is flexible then the 2-functor $\TAlgs(A, \thg)$ sends $\us$-injective equivalences to injective equivalences}.  It seems unlikely that the converse of this implication holds in general; but it does at least in the case of weights, as the following theorem shows.
\begin{Theorem}\label{thm:strangeweights}
A weight $W \in [\f J, \Cat]$ is flexible if and only if for all complete $2$-categories $\f K$, the $2$-functor $\{W, \thg\}  \colon [\f J, \f K] \to \f K$ sends pointwise injective equivalences to injective equivalences.
\end{Theorem}
\begin{proof}
If $W$ is flexible then for each complete $\f K$ the limit 2-functor $\{W, \thg\}\colon [\f J, \f K] \to \f K$ admits, by Theorem~\ref{thm:char2weight}, an extension $H \colon \Ps(\f J, \f K) \to \f K$ preserving identities strictly and satisfying $H(g \circ f)=H(g) \circ H(f)$ whenever $f$ or $g$ is strict.  Just as in the preceding discussion, an extension of this form ensures that $\{W, \thg\}  \colon [\f J, \f K] \to \f K$ sends pointwise injective equivalences to injective equivalences.

Conversely suppose the limit 2-functor $\{W, \thg\}  \colon [\f J, \f K] \to \f K$ has this property for each complete $\f K$.  This is easily seen to be equivalent to the assertion that the colimit 2-functor $W \star \thg \colon [\f J^{\op}, \f K] \to \f K$ sends pointwise surjective equivalences to surjective equivalences for each \emph{cocomplete} $\f K$.  Let $\f K$ be the cocomplete 2-category $[\f J,\Cat]$ and consider the comonad $Q$ thereon, induced by the inclusion $\iota \colon [\f J,\Cat] \to \Ps(\f J,\Cat)$ and its left adjoint. That left adjoint exists on replacing $\Cat$ by any cocomplete $2$-category; in particular, on replacing $\Cat$ by $\Cat^\op$, from which it follows that $\iota$ admits also a \emph{right} $2$-adjoint; this was observed in~\cite[Section~4.2]{Gambino2008Homotopy}. Thus the comonad $Q$ is a composite of left adjoint $2$-functors and as such is cocontinuous.

Now by definition a component of the counit $p \colon Q \to 1$ at a weight is a surjective equivalence just when that weight is flexible.  Each representable is flexible so that precomposing $p$ by the Yoneda embedding $Y \colon \f J^{\op} \to [\f J,\Cat]$ yields a pointwise surjective equivalence $p \circ Y \colon Q \circ Y \to Y$ in $[\f J^{\op},[\f J,\Cat]]$.  Taking the image of $p \circ Y$ under $W \star \thg \colon [\f J^{\op}, [\f J,\Cat]] \to [\f J, \Cat]$ we consequently obtain a surjective equivalence $W \star (Q \circ Y) \to W \star Y$.  On the one hand the Yoneda lemma asserts that $W \star Y \cong W$; on the other hand cocontinuity of $Q$ ensures that $W \star (Q \circ Y) \cong Q(W \star Y) \cong QW$. Combining these isomorphisms with the above surjective equivalence now yields a surjective equivalence $QW \to W$; whence $W$, as a retract of $QW$, is flexible.
\end{proof}

\section{Closure properties}\label{sec:closure}
In this section, we continue our study of the semiflexible, flexible and pie algebras for a $2$-monad $T$ by considering the closure properties of these three classes of algebras. Once again, we remind the reader of our standing assumptions that $T$ be a $2$-monad with rank on a complete and cocomplete $2$-category $\f C$. 

\begin{Proposition}\label{prop:pie}
The semiflexible, flexible and pie algebras each contain the frees, and are each closed in $\TAlgs$ under the corresponding class of weighted colimits.
\end{Proposition}
\begin{proof}
The free-forgetful adjunction $\fs \colon \f C \leftrightarrows \TAlgs \colon \us$ generates the comonad $\fs \us$ on $\TAlgs$, and clearly every free $T$-algebra bears $\fs \us$-coalgebra structure. But as in Corollary~\ref{cor:comonadmap}, we have a comonad map $\fs \us \to Q$, whence every free $T$-algebra bears strict $Q$-coalgebra structure, and is thus pie, flexible and semiflexible.

As for the closure under colimits, we consider only the semiflexible case, the arguments in the other two cases being identical in form. So let $W \in [\f J, \Cat]$ be a semiflexible weight, and $D \colon \f J^\op \to \TAlgs$ a $2$-functor taking values in semiflexible algebras; we must show that the colimit $W \star D$ in $\TAlgs$ is again semiflexible. By Theorem~\ref{thm:char3}, it suffices for this to show that $\TAlgs(W \star D, \thg) \colon \TAlgs \to \Cat$ takes $U$-equivalences to equivalences. Now by definition of weighted colimit, we have a $2$-natural isomorphism
\begin{equation*}
\TAlgs(W \star D, \thg) \cong [\f J, \Cat](W?, \TAlgs(D?, \thg))\rlap{ ,}
\end{equation*}
and so it is enough to show that the $2$-functor on the right-hand side takes $U$-equivalences to equivalences. But as each $DX$ is semiflexible, the functor $\TAlgs(D?, \thg) \colon \TAlgs \to [\f J, \Cat]$ takes $U$-equivalences to pointwise equivalences; and as $W$ is semiflexible, the functor $[\f J, \Cat](W, \thg) \colon [\f J, \Cat] \to \Cat$ takes pointwise equivalences to equivalences. Combining these two facts yields the result.
\end{proof}
Specialising to the case of weights yields the following corollary. In light of Proposition~\ref{prop:pieweights}, the case dealing with pie weights is tautologous; the one for flexibles is Theorem~4.9 of~\cite{Bird1989Flexible}, whilst the semiflexible case appears to be new. What is certainly new is the possibility of treating all three cases in a uniform manner.
\begin{Corollary}
Each of the classes of semiflexible, flexible and pie weights is saturated.
\end{Corollary}
\begin{proof}
As in~\cite{Albert1988The-closure}, to say that a class $\Phi$ of weights is saturated is to say that the
full subcategory of $[\f J, \Cat]$ spanned by the weights in $\Phi$ with domain $\f J$ contains the representables and is closed under $\Phi$-colimits. If we view $[\f J, \Cat]$ as the category of algebras for the weight $2$-monad on $[\mathrm{ob}\ \f J, \Cat]$, then every representable weight is free as a $T$-algebra; whence the result follows from the preceding proposition.
\end{proof}
The preceding result can be restated as saying that, in the case of the $2$-monad for weights, the semiflexible, flexible and pie algebras in $[\f J, \Cat]$ are precisely the closure of the free algebras in $\TAlgs$ under the corresponding class of colimits. In the remainder of this section, we consider the extent to which this remains true on replacing weights by algebras for a general $2$-monad $T$.
For the flexible algebras, the answer is straightforward; the following was observed in Remark 5.5 of~\cite{Lack2007Homotopy-theoretic}.
\begin{Proposition}
The flexible algebras are the closure of the frees in $\TAlgs$ under flexible colimits.
\end{Proposition}
\begin{proof}
As in~\cite{Lack2002Codescent}, every algebra of the form $QA$ may be presented as a pie colimit of free algebras. Now every flexible algebra is a retract of one of the form $QA$, and thus the splitting of an idempotent on some $QA$. Since the colimit which splits an idempotent is flexible, every flexible algebra is a flexible colimit of frees.
\end{proof}
There is an equally straightforward statement concerning the semiflexibles.
\begin{Proposition}
The semiflexible algebras are the closure of the frees in $\TAlgs$ under bicolimits.
\end{Proposition}
\begin{proof}
First observe that the semiflexibles are closed under bicolimits. For indeed, if $X$ is a $W$-weighted bicolimit of a diagram $D$ of semiflexible algebras, then $X$ is equivalent to $W\star_\mathrm{ps}D$, the $W$-weighted pseudocolimit of $D$; since semiflexibles are closed under equivalence it is therefore enough to show that $W\star_\mathrm{ps}D$ is semiflexible. But $W\star_\mathrm{ps}D
\cong QW \star D$ is a semiflexible colimit of semiflexibles, and so indeed semiflexible.
Every free algebra is semiflexible, and so to complete the proof it suffices to show that every semiflexible is a bicolimit of frees. As we observed above, any algebra of the form $QA$ is a pie colimit of frees; hence a semiflexible colimit of frees, and so, by (the dual of) Corollary~\ref{cor:semiflexbi}, a bicolimit of frees. An algebra is semiflexible just when it is equivalent to one of the form $QA$, and any object equivalent to a bicolimit of frees is again a bicolimit of frees.
\end{proof}
We have not attempted to ascertain whether or not the semiflexible algebras are, in fact, the closure of the frees in $\TAlgs$ under semiflexible colimits; our investigations have shown only that there seems to be no obvious proof of this fact. We have, however, considered in some detail the corresponding question for the pie algebras, and found it to be false: the pie algebras for an arbitrary $2$-monad $T$ may comprise a strictly larger class than the closure of the frees in $\TAlgs$ under pie colimits. We will describe a $2$-monad exhibiting this divergence in the proof of Proposition~\ref{prop:counterexample} below.

However, it turns out that the two classes just mentioned \emph{will} coincide under certain additional assumptions on our $2$-monad $T$ and our $2$-category $\f C$. The remainder of this section will be given over to the analysis underlying this result. The key step will be to interpose between the pie algebras and the pie colimits of the frees a third class of algebras: the objective quotients of the frees.  Here, we call an object $B$ of a $2$-category an \emph{objective quotient} of an object $A$ if there exists some objective morphism $A \to B$.
\begin{Proposition}\label{thm:classesofpie}
Considering the following three classes of $T$-algebras:
\begin{enumerate}[(1)]
\item The closure of the frees  under pie colimits;
\item The objective quotients  of the frees; and
\item The pie algebras,
\end{enumerate}
we have inclusions (1) $\subseteq$ (2) $\subseteq$ (3).
\end{Proposition}
In the statement of this result, and in what follows, when we speak of colimits or of objective quotients, it is to be understood that these are being taken in $\TAlgs$.
\begin{proof}
To prove (1) $\subseteq$ (2) it is clearly enough to show that the objective quotients of the frees are closed under pie colimits. For closure under coproducts, suppose that $(A_i \mid i \in I)$ is a collection of $T$-algebras which are objective quotients of frees, and let $q_i \colon FX_i \to A_i$ witness this fact for each $i \in I$. Now $\Sigma_i q_i \colon \Sigma_i \fs X_i \to \Sigma_i A_i$ is objective, since objective morphisms are closed under colimits in the arrow category, and $\Sigma_i \fs X_i$ is free on $\Sigma_i X_i$, whence $\Sigma_i A_i$ is an objective quotient of a free.
For closure under coinserters, let $f, g \colon A \rightrightarrows B$ in $\TAlgs$, with $A$ and $B$ objective quotients of frees. Now the coinserter morphism $h \colon B \to C$, like any coinserter morphism, is objective; whence $C$ is an objective quotient of $B$, and hence also an objective quotient of a free. The case of coequifiers is identical on observing that coequifier morphisms, too, are objective.

As for (2) $\subseteq$ (3), we saw in Proposition~\ref{prop:pie} above that every free algebra is a pie algebra; the proof is now completed by the following proposition, which we separate out for its independent interest.\end{proof}

\begin{Proposition}
Any objective quotient of a pie algebra is a pie algebra.
\end{Proposition}
\begin{proof}
Let $A \in \TAlgs$ admit strict $Q$-coalgebra structure $a \colon A \to QA$, and let $f \colon A \to B$ be an objective morphism. We must show that $B$ admits strict $Q$-coalgebra structure. So consider the square on the left in the diagram:
\begin{equation*}
\cd{
  A \ar[d]_{Qf \circ a} \ar[r]^-{f} & B \ar[d]^{1_B} \ar@{-->}[dl]^{b} \\
  QB \ar[r]_{p_B} & B 
} \qquad \qquad
\cd{
  A \ar[d]_{\Delta
  \circ Qf \circ a} \ar[r]^-{f} & B \ar[d]^{b} \\
  QQB \ar[r]_{p_{QB}} & QB\rlap{ .}
}\end{equation*}
Since $f$ is objective, and $p_B$ fully faithful, there exists as indicated a unique diagonal filler $b \colon B \to QB$ making both triangles commute. We claim $b$ equips $B$ with $Q$-coalgebra structure. The lower-right triangle asserts the counit axiom $p_B \circ b = 1_B$; as for the comultiplication axiom, it is easy to verify that both $Qb \circ b$ and $\Delta \circ b$ are diagonal fillers for the right-hand square above, and so must be equal.
\end{proof}

We now turn our attention to finding conditions under which the two inclusions of Proposition~\ref{thm:classesofpie} become equalities. It seems that obtaining necessary and sufficient conditions is a hard problem; we shall therefore content ourselves with giving sufficient conditions that are broad enough to encompass the classes of examples that we studied in Section~\ref{sec:exs}.  We first consider when the inequality (2) $\subseteq$ (3) becomes an equality; this is closely related to our Theorem~\ref{thm:charthm}, and in what follows, we make free use of the definitions and notational conventions established for that result. In particular, we recall from~\eqref{eq:situation} the comparison functor $j \colon \ToAlg \to \TpAlg$, existing for a $\f C$ with enough discretes, that assigns to  each $T$-algebra its induced structure ``at the level of objects''.

\begin{Proposition}\label{prop:objproj}
If $\f C$ has enough discretes and (objective, fully faithful) factorisations lifting to $\TAlgs$, then for each $T$-algebra $A$, the following are equivalent:
\begin{enumerate}[(i)]
\item $A$ is an objective quotient of a free;
\item $A$ is an objective quotient of a free on a (projective) \proj;
\item $jA$ is a free $\Tp$-algebra.
\end{enumerate}
\end{Proposition}
\begin{proof}
First we prove (i) $\Leftrightarrow$ (ii). For the non-trivial direction,  if $q \colon \fs X \to A$ exhibits $A$ as an objective quotient of a free, then $q \circ \fs \lambda_X \colon \fs DOX \to A$ exhibits it as an objective quotient of a free on a \proj: observe that $\fs \lambda_X$ is objective in $\TAlgs$ since $\lambda_X$ is objective in $\f C$ and left adjoints preserve objectivity.

Now we prove (ii) $\Leftrightarrow$ (iii). For any $X \in \Cp$, since $j\fs DX = \fp X$, the action of $j$ on homs gives a function $\ToAlg(\fs DX, A) \to \TpAlg(\fp X, jA)$; this is in fact a bijection, with inverse given by the composite
\[
 \TpAlg(\fp X, jA) \cong
 \Cp(X, \up jA) =
 \Cp(X, O \us A) \cong
 \ToAlg( \fs D X, A)\rlap{ .}
\]
We claim that under this bijection, objective morphisms $\fs DX \to A$ correspond with isomorphisms $\fp X \to jA$; this will complete the proof. For this, we observe that $q \colon \fs DX \to A$ is objective just when $O\us q = \up j q$ is invertible in $\f A$; but $\up$ is conservative, so this happens just when $jq$ is invertible in $\TpAlg$, as desired.
\end{proof}

Combining this result with Theorem~\ref{thm:charthm}, we obtain:

\begin{Corollary}\label{cor:2in3}
Let $\f C$ have enough \projs\ and (objective, fully faithful) factorisations lifting to $\TAlgs$. If $\Cp$ has, and the induced monad $\Tp$ preserves, coreflexive equalisers, then every pie $T$-algebra is an objective quotient of a free on a discrete.
\end{Corollary}

As we observed in Section~\ref{sec:pie}, a free algebra on a projective discrete is again projectively discrete. Thus the preceding corollary ensures that, in particular, the full sub-$2$-category of $\TAlgs$ spanned by the pie algebras has enough discretes, and that, in fact, the projectively discrete objects of this $2$-category are precisely the algebras free on projective discretes. 

We now consider when the inequality (1) $\subseteq$ (2) of Proposition~\ref{thm:classesofpie} becomes an equality; we shall do so in a slightly less general situation than that we have considered so far.

\begin{Proposition}\label{thm:1in2}
Let $\f J$ be a small $2$-category, and $T$ an objective-preserving $2$-monad on $[\f J, \Cat]$. The following are equivalent properties of the $T$-algebra $A$:
\begin{enumerate}[(i)]
\item $A$ is an objective quotient of a free;
\item $A$ is an objective quotient of a free on a discrete;
\item $A$ lies in the closure of the frees under pie colimits;
\item $A$ lies in the closure of the frees on discretes under pie colimits.
\end{enumerate}
\end{Proposition}
\begin{proof}
We have (iv) $\Rightarrow$ (iii) trivially, (iii) $\Rightarrow$ (i) by Proposition~\ref{thm:classesofpie} and (i) $\Rightarrow$ (ii) by Proposition~\ref{prop:objproj}. It remains to show that (ii) $\Rightarrow$ (iv). First note that by Proposition~\ref{prop:whenpresbij}, since $T$ preserves objectives so too does $U \colon \TAlgs \to [\f J, \Cat]$. Thus it is clearly sufficient that we prove:

\emph{If $q_0 \colon X_0 \to A$ in $\TAlgs$ has $Uq_0$ pointwise bijective on objects, and $X_0$ is a pie colimit of frees on discretes, then so is $A$}. To see this, form the comma object in $[\f J, \Cat]$ as on the left of the diagram:
\begin{equation*}
\cd{
Y \ar[r]^d \ar[d]_c \dtwocell{dr}{\gamma} & \us X_0 \ar[d]^{\us q_0} \\
\us X_0 \ar[r]_{\us q_0} & \us A
} \qquad  \qquad
\cd{
\fs DOY \ar[r]^{\bar d} \ar[d]_{\bar c} \dtwocell{dr}{\delta} & X_0 \ar[d]^{i_0} \\
X_0 \ar[r]_{i_0} & X_1\rlap{ .}
}\end{equation*}
Composing $d$ and $c$ with the discrete cover $\lambda_Y \colon DOY \to Y$ and taking transposes, we obtain maps $\bar d, \bar c \colon \fs DOY \rightrightarrows X_0$; we now form
their coinserter as on the right above. The transpose of the $2$-cell $\gamma \circ \lambda_Y$ under adjunction constitutes a coinserter cocone under $\bar d, \bar c$ and so we have an induced map $q_1 \colon X_1 \to A$ making the triangle
\begin{equation*}
\cd{ X_0 \ar[rr]^{i_0} \ar[dr]_{q_0} & & X_1 \ar[dl]^{q_1} \\ & A}
\end{equation*}
commute. Now $i_0$ is objective, being a coequifier morphism, and so $Ui_0$, like $Uq_0$, is pointwise bijective on objects; which in turn implies that $Uq_1$ is pointwise bijective on objects.
We claim that it is also pointwise full. Indeed, given objects $x$ and $y$ in $X_1(j)$ and a morphism $f \colon q_1x \to q_1y$ in $A(j)$, we have, since $Ui_0$ is pointwise bijective on objects, a cone as on the left in the diagram:
\begin{equation*}
\cd{
\f J(j, \thg) \ar[r]^{i_0^{-1}x} \ar[d]_{i_0^{-1}y} \dtwocell{dr}{f} & \us X_0 \ar[d]^{\us q_0} \\
\us X_0 \ar[r]_{\us q_0} & \us A
} \qquad  \qquad
\cd{
\f J(j, \thg) \ar[r]^-{\bar f} & DOY \ar[r]^-{\eta} & \us \fs DOY \ar@/^12pt/[r]^{U(i_0 \bar d)} \ar@/_12pt/[r]_{U(i_0 \bar c)} \dtwocell[0.43]{r}{\us \delta} & \us X_1\rlap{ .}
}\end{equation*}
This cone corresponds to a map $\f J(j,
\thg) \to Y$ which, since $\f J(j, \thg)$ is projectively discrete, factors through the objective $\lambda_Y \colon DOY \to Y$ as $\bar f \colon \f J(j, \thg) \to DOY$, say.
So we have a $2$-cell as on the right above; now this in turn corresponds to a map $g \colon x \to y$ in $X_1(j)$, and it is easy to verify that $q_1g = f$ as required. Thus we have shown that $Uq_1$ is pointwise bijective on objects and full, and so to complete the proof, it will suffice to show that:

\emph{If $q_1 \colon X_1 \to A$ in $\TAlgs$ has $Uq_1$ pointwise bijective on objects and full, and $X_1$ is a pie colimit of frees on discretes, then so is $A$}. To this end, let us write $\mathbf 2$ for the arrow category, and $\f P$ for the category $\bullet \rightrightarrows \bullet$, as in Section~\ref{sec:2catsrelated}. Taking $K \colon \f P \to \mathbf 2$ to be the unique bijective-on-objects functor, we now form in $\f C$ the limit of the arrow $Uq_1$ weighted by $K \in [\mathbf 2, \Cat]$, as on the left of the diagram:
\begin{equation*}
\cd[@C+1em]{
  Z \ar@/^12pt/[r]^{u} \ar@/_12pt/[r]_{v} \dtwocell[0.32]{r}{\sigma} \dtwocell[0.6]{r}{\tau} &
  UX_1 \ar[r]^{\us q_1} & UA
} \qquad \qquad
\cd[@C+1em]{
  \fs DOZ \ar@/^12pt/[r]^{\bar u} \ar@/_12pt/[r]_{\bar v} \dtwocell[0.35]{r}{\bar \sigma} \dtwocell[0.62]{r}{\bar \tau} &
  X_1 \ar[r]^{i_1} & X_2\rlap{ .}
}
\end{equation*}
Thus $\us q_1 \circ \sigma = \us q_1 \circ \tau$ and $(Z, u, v, \sigma, \tau)$ are universal amongst such data; we call them the \emph{equi-kernel} of $Uq_1$. Whiskering with $\lambda_Z \colon DOZ \to Z$ and taking transposes, we obtain $\bar u$, $\bar v$, $\bar \sigma$ and $\bar \tau$ as on the right; let $i_1$ as displayed be their coequifier. Since clearly $q_1 \circ \bar \sigma = q_1 \circ \bar \tau$, we obtain an induced morphism $q_2 \colon X_2 \to A$ making the triangle
\begin{equation*}
\cd{ X_1 \ar[rr]^{i_1} \ar[dr]_{q_1} & & X_2 \ar[dl]^{q_2} \\ & A}
\end{equation*}
commute. Now since $i_{1}$ is a coequifier morphism it is objective; as $q_{1}$ is objective so also is $q_{2}$ and thus $Uq_{2}$ is pointwise bijective on objects.  We claim that it is also a split epimorphism. To see this, observe first that the equality $i_1 \circ \bar \sigma = i_1 \circ \bar \tau$ transposes under adjunction to give $Ui_1 \circ \sigma \circ \lambda_{Z} = Ui_1 \circ \tau \circ \lambda_{Z}$; now as $\lambda_{Z}$ is pointwise bijective on objects it is cofaithful, and thus $Ui_{1} \circ \sigma = Ui_{1} \circ \tau$. Since $\us q_1$ is pointwise bijective on objects and full, it is by an easy calculation the coequifier of its own equi-kernel $(\sigma, \tau)$, and so there is a unique map $g \colon \us  A \to UX_2$ with $g \circ \us  q_1 = \us  i_1$.  Now $\us q_2 \circ g \circ Uq_1 = \us q_2 \circ \us i_1 = \us q_1$ and hence $Uq_2 \circ g = 1_{\us A}$, since $\us q_1$ is pointwise bijective on objects and full, and thus epimorphic.
Thus $\us q_2$ is a split epimorphism as claimed; since it is also pointwise bijective on objects, it is therefore pointwise full. But now also $UFUq_2$ is a split epimorphism, and also pointwise bijective on objects, since $T = UF$ preserves objectives; whence $UFUq_2$ is also pointwise bijective on objects and full. To complete the proof it will therefore suffice to show that:

\emph{If $q_2 \colon X_2 \to A$ in $\TAlgs$ has both $Uq_2$ and $UFUq_2$ pointwise bijective on objects and full, and $X_2$ is a pie colimit of frees on discretes, then so is $A$}. To prove this, we form the equi-kernel of $\us q_2$, as on the left of the diagram:
\begin{equation*}
\cd[@C+1em]{
  W \ar@/^12pt/[r]^{w} \ar@/_12pt/[r]_{z} \dtwocell[0.32]{r}{\phi} \dtwocell[0.6]{r}{\psi} &
  \us X_2 \ar[r]^{\us q_2} & \us A
} \qquad \qquad
\cd[@C+1em]{
  \fs DOW \ar@/^12pt/[r]^{\bar w} \ar@/_12pt/[r]_{\bar z} \dtwocell[0.35]{r}{\bar \phi} \dtwocell[0.62]{r}{\bar \psi} &
  X_2 \ar[r]^{q_2} & A\rlap{ .}
}
\end{equation*}
Whiskering with $\lambda_W \colon DOW \to W$ and taking transposes, we obtain $\bar w$, $\bar z$, $\bar \phi$ and $\bar \psi$ as on the right; we now claim that $q_2$ is a coequifier of $\bar \phi$ and $\bar \psi$, which will complete the proof. Observe first that $q_2$ is epimorphic in $\TAlgs$, since $U$ is faithful and $Uq_2$ is epimorphic in $[\f J, \Cat]$; so it suffices to show that any map $g \colon X_2 \to B$ of $T$-algebras satisfying $g \circ \bar \phi = g \circ \bar \psi$ factors through $q_2$.

For such a $g$, by taking adjoint transposes we see that $\us g \circ \phi \circ \lambda_W = \us g \circ \psi \circ \lambda_W$, and since $\lambda_W$ is cofaithful, it follows that $\us g \circ \phi = \us g \circ \psi$. Since $\us q_2$ is bijective on objects and full, it is as before the coequifier in $[\f J, \Cat]$ of $\phi$ and $\psi$, and hence $\us g$ admits a factorisation $\us g = h \circ \us q_2$. We will be done if we can show that $h \colon \us A \to \us B$ is in fact an algebra map; or in other words, that the square on the right of the diagram
\begin{equation*}
\cd[@C+2em]{
\us \fs\us X_2 \ar[d]_{\us \epsilon_{X_2}} \ar[r]^{\us \fs\us q_2} & \us \fs\us A \ar[d]_{\us  \epsilon_{A}} \ar[r]^{\us \fs h} & \us B \ar[d]^{\us \epsilon_B} \\
\us X_2 \ar[r]_{\us q_2} & \us A \ar[r]_{h} & \us B}
\end{equation*}
is commutative. But since $\us \fs\us q_2$ is epimorphic, the commutativity of the right-hand square follows from that of the outside and of the left-hand square.\end{proof}
Taking together Theorem~\ref{thm:charthm}, Corollary~\ref{cor:2in3} and Proposition~\ref{thm:1in2}, we obtain the main result of this section:
\begin{Theorem}\label{thm:classescoincide}
Let $\f J$ be a small $2$-category and $T$ an objective-preserving $2$-monad on $[\f J, \Cat]$ for which the induced monad $T_d$ on $[\f J_0, \Set]$ preserves coreflexive equalisers. Then the following classes of $T$-algebras coincide:
\begin{enumerate}
\item The closure of the frees (on discretes) under pie colimits;
\item The objective quotients of the frees (on discretes);
\item The pie algebras;
\end{enumerate}
and may be characterised as the $T$-algebras $A$ for which $jA$ is a free $\Tp$-algebra.
\end{Theorem}

Before continuing, let us make a few observations concerning the proof of Proposition~\ref{thm:1in2}.  In that proof we only made use of the fact that the base 2-category was $[\f J,\Cat]$ in two places: in determining that the objective $\lambda_{W}$ is cofaithful, and in determining that $Uq_{1}$ is full. Regarding the first of these, objectives are in fact cofaithful in any sufficiently complete $2$-category; as for the second, we only exploited the fullness of $Uq_{1}$ in later asserting it to be the coequifier of its equikernel, and this assertion makes sense in any $2$-category.  It is possible then to at least state the argument of this proof over other bases, and it would be valid so long as the base $2$-category in question had enough exactness properties (and discretes).  Since there is no concerted account of $2$-categorical exactness properties in print, we have chosen to work in the more concrete setting of a presheaf $2$-category. 

There is however another set of conditions, more directly related to exactness notions, under which the pie colimits of frees and the objective quotients of frees can be shown to coincide. These were investigated in the first author's \cite{Bourke2010Codescent}; they allow \f{C} to be more general than a presheaf 2-category but impose stronger conditions on $T$ than its preserving objectives.  Since these conditions are more difficult to verify in our examples, we have found no cause to give the arguments in any detail here; but let us nonetheless break off briefly to discuss them.

Given a map of sets $f \colon A \to B$, we may obtain its (strong epi, mono) factorisation by first forming the kernel-pair of $f$, and then the coequaliser $f_1 \colon A \to C$ of that. Since $f$ also coequalises this kernel-pair, it must factor through $f_1$ as $f = f_2 \circ f_1$; now $f_1$ is regular epi, hence strong epi, and it turns out that $f_2$ is mono, so giving the required factorisation.
We have a similar limit-colimit description of the (objective, fully faithful) factorisation of a functor $f \colon A \to B$.  By taking the comma object $f|f$ and appropriate pullbacks one may form the \emph{higher kernel} of $f$, which is an internal category in \Cat as on the left below:
$$\xy
(0,0)*+{f|f|f}="c0"; (20,0)*+{f|f}="b0";(40,0)*+{A}="a0";(70,0)*+{B\rlap{ .}}="d0";(55,-10)*+{C}="e0";
{\ar@<1.5ex>^{d} "b0"; "a0"}; 
{\ar@<0ex>|{i} "a0"; "b0"}; 
{\ar@<-1.5ex>_{c} "b0"; "a0"}; 
{\ar@<1.5ex>^{p} "c0"; "b0"}; 
{\ar@<0ex>|{m} "c0"; "b0"}; 
{\ar@<-1.5ex>_{q} "c0"; "b0"};
{\ar@<0ex>_{f_{1}} "a0"; "e0"};
{\ar@<0ex>_{f_{2}} "e0"; "d0"};
{\ar@<0ex>^{f} "a0"; "d0"};
\endxy$$

A \emph{codescent cocone} under this higher kernel is given by a $1$-cell $k \colon A \to X$ together with a $2$-cell $\gamma \colon kd \Rightarrow kc$ satisfying cocycle conditions. In $\Cat$, there is a universal such cocone $(f_1 \colon X \to C, \gamma)$, with $C$ called the \emph{codescent object} of the given data, and $f_1$ a \emph{codescent morphism}. Now $f$ itself forms part of a codescent cocone under its higher kernel, whence there is an induced factorisation $f = f_2 \circ f_1$ as on the right above. Just as each regular epi is strong epi, so every codescent morphism is objective, so in particular $f_1$ is objective; moreover, it turns out that in $\Cat$, the induced $f_2$ is always fully faithful, so that we obtain an (objective, fully faithful) factorisation of $f$ as claimed.

One can emulate this same argument in any 2-category $\f C$ which admits higher kernels and codescent objects of higher kernels, so long as when one factorises $f = f_2 \circ f_1$ in $\f C$ as above,  the arrow $f_2$ is known always to be fully faithful; this is an exactness property of the $2$-category $\f C$, analogous to regularity in the one-dimensional setting. In such a $\f C$ we have (codescent, fully faithful) factorisations, and it follows that the objectives and codescent morphisms in $\f C$ coincide. Now given a 2-monad $T$ on such a $\f C$ \emph{preserving codescent objects of higher kernels} the factorisation described above lifts, so that objectives and codescent morphisms in \TAlgs again coincide; and under these assumptions, together with our standing ones, any objective quotient of a free $T$-algebra admits a canonical presentation as a codescent object, which exhibits it as belonging to the closure of the frees under pie colimits. For the details of this argument, see~\cite[Chapter 9]{Bourke2010Codescent}.

We conclude this section with an example showing, as promised, that the classes of algebras in Proposition~\ref{thm:classesofpie} need not always coincide.
\begin{Proposition}\label{prop:counterexample}
Not every pie algebra need lie in the closure of the frees under pie colimits.
\end{Proposition}
\begin{proof}
Consider the 2-monad $T$ on \Cat whose action on objects is given by $T(0)=0$ and  $TC = C + 1$ whenever $C$ is non-empty. An algebra is either a pointed category, or the empty category, and an algebra map is either a functor preserving the point, or the unique functor out of the empty category. Clearly $T$ preserves objectives, which are just the bijections on objects; thus the objectives in \TAlgs are exactly the bijective on objects algebra maps.  By Proposition~\ref{thm:classesofpie} each $T$-algebra $A$ lying in the closure of the frees under pie colimits admits an objective morphism $FC \to A$ from some free $T$-algebra; thus the sets of objects of $A$ and $FC$ have the same cardinality.  By definition of $T$ we know that the underlying category of $FC$ has either no objects (if $C$ is the empty category) or at least two; in particular the underlying category of such an $A$ cannot have a single object.

Now the terminal category $1$, which does have a single object, admits a unique $T$-algebra structure;  so we will be done if we can show that this $T$-algebra is pie. By Proposition~\ref{prop:pieproj}, this will be so if and only if $j(1)$, the $1$-element set with its unique point, admits an $F_{d}U_{d}$-coalgebra structure for the induced monad $T_{d}$ on $\Set$. Arguing as before, the algebras for $T_d$ are either pointed sets or the empty set; and amongst these, $j(1)$ is initial with respect to all but the empty one.  As $F_{d}U_{d}j(1)$ is non-empty, having cardinality $2$, there exists a unique $T_{d}$-algebra morphism $j(1) \to F_{d}U_{d}j(1)$.  To check that this equips $j(1)$ with a coalgebra structure involves checking the commutativity of diagrams with codomains $j(1)$ and $(F_{d}U_{d})^{2}j(1)$, both of which are non-empty algebras; since $j(1)$ is initial with respect to such algebras both diagrams necessarily commute, so that $j(1)$ admits $F_{d}U_{d}$-coalgebra structure; whence $1$ with its point is a pie algebra.
\end{proof}

\section{Strongly finitary $2$-monads}\label{sec:sf2monad}

In this final section, we give an extended application of the results developed in the rest of the paper. We will use them to characterise, amongst the $2$-monads on $\Cat$ of a certain class, those whose algebras can be presented as categories equipped with basic operations and basic natural transformations between derived operations, satisfying equations between derived natural transformations but with no equations being imposed between derived operations themselves. Let us agree to call such a $2$-monad \emph{pie-presentable}; a precise definition will be given shortly. 

At various places in the literature can be found observations to the effect that each pie-presentable $2$-monad is  a flexible $2$-monad in the sense of~\cite{Kelly1974Doctrinal}; see, for example, \cite[Proposition 4.3]{Kelly1974Coherence}, \cite[Remark 6.3]{Kelly2004Monoidal} or~\cite[Section 7.5]{Lack2007Homotopy-theoretic}. Under suitable cardinality constraints, the notion of flexible $2$-monad in question can, as in the introduction, be seen as a particular case of the general one: the $2$-category of monads with rank $\kappa$ on $\Cat$ is $2$-monadic over the corresponding $2$-category of endofunctors, and the flexible algebras for the induced $2$-monad are precisely the flexible $2$-monads. 

In this situation, we have also the corresponding pie algebras, hence a notion of pie $2$-monad. It turns out that every pie-presentable $2$-monad is not just flexible but also pie, but that this still does not characterise the pie-presentables amongst all $2$-monads. For example, consider the $2$-monad on $\Cat$ for which an algebra is a category $C$ equipped with a strictly commutative binary operation $\theta \colon C \times C \to C$. This $2$-monad $T$ is free on the endofunctor $\mathrm{Sym}^2$ of $\Cat$ 
which sends a category $C$ to the coequaliser of the identity and symmetry maps $C \times C \to C \times C$; thus $T$ is a pie $2$-monad by Proposition~\ref{prop:pie}, but it is not pie-presentable by virtue of the equation between operations $\theta(\thg,?) = \theta(?, \thg)$ which must be imposed. In characterising the pie-presentable $2$-monads we must therefore make one further refinement. We consider $2$-monads of our chosen class as $2$-monadic over a $2$-category of \emph{signatures}, and call the pie algebras for the induced $2$-monad the \emph{strongly pie} $2$-monads. Our result now states that, amongst the class of $2$-monads we will consider, a $2$-monad is pie-presentable if and only if it is strongly pie.

We will prove this result using the tools developed in the previous section; and in order for these to be applicable, we must ensure that when we view our class of $2$-monads as $2$-monadic over signatures,  the induced $2$-monad preserves objectives. This would \emph{not} be the case if, for example,  we were to consider all finitary $2$-monads on $\Cat$; and for this reason, we shall instead be concerned with the \emph{strongly finitary} $2$-monads of~\cite{Kelly1993Finite-product-preserving}, which are roughly speaking those whose algebras may be defined only with reference to functors of the form $\f C^n \to \f C$ for various natural numbers $n$, and to natural transformations between such functors. Whilst this rules out such $2$-monads as that for categories with finite limits---which requires, amongst other things, an operation of the form $\f C^{\cdot \rightrightarrows \cdot} \to \f C$ to express the taking of equalisers---we still retain an acceptably large class of examples including the $2$-monads for monoidal categories, categories with finite products, distributive categories, and so on. The characterisation result we will prove, then, is that a strongly finitary $2$-monad on $\Cat$ is pie-presentable if and only if it is strongly pie. Of course, there is a corresponding version of this result where ``finitary'' has been replaced by ``of rank $\kappa$'' for some regular cardinal $\kappa$; we leave its formulation to the reader.

We now define precisely the notions appearing in the statement of our result. Let $\mathbb F$ denote the full sub-$2$-category of $\Cat$ spanned by the categories $0, 1, 2, \dots$ discrete on the corresponding number of elements. The inclusion $\mathbb F \to \Cat$ induces by left Kan extension a $2$-fully faithful functor $[\mathbb F, \Cat] \to \End(\Cat)$, left adjoint to restriction; and as in~\cite{Kelly1993Finite-product-preserving}, an endofunctor of $\Cat$ is called \emph{strongly finitary} when it is in the essential image of this $2$-functor. The $2$-category $\Endsf(\Cat)$ of strongly finitary endofunctors is thus $2$-equivalent to $[\mathbb F, \Cat]$, and so locally finitely presentable. It was shown in~\cite{Kelly1993Finite-product-preserving} that the strongly finitary endofunctors are closed under composition so that the inclusion $2$-functor $\Endsf(\Cat) \to \End(\Cat)$ is strict monoidal. It is also a left adjoint, since $[\mathbb F, \Cat] \to \End(\Cat)$ is so, and thus by ~\cite{Kelly1974Doctrinal} is the left adjoint part of a monoidal adjunction: as such, it lifts to a left adjoint $\Mndsf(\Cat) \to \Mnd(\Cat)$ between the corresponding 2-categories of monoids. On the right is the $2$-category of $2$-monads on $\Cat$; on the left is the 2-category of \emph{strongly finitary 2-monads}, which are equally just the 2-monads whose underlying endofunctor is strongly finitary.  There is a forgetful $2$-functor $W \colon \Mndsf(\Cat) \to \Endsf(\Cat)$ which has a left adjoint and is finitarily $2$-monadic; it follows that $\Mndsf(\Cat)$, like $\Endsf(\Cat)$, is locally finitely presentable, and in particular, complete and cocomplete. 

Let us write $\mathbb N$ for the discrete sub-$2$-category of $\Cat$ spanned by the categories $0, 1, 2, \dots$; by a \emph{signature}, we mean an object of the presheaf $2$-category $[\mathbb N, \Cat]$. There is a forgetful $2$-functor $V \colon \Endsf(\Cat) \to [\mathbb N, \Cat]$, given by restriction along the inclusion $\mathbb N \hookrightarrow \Cat$, and this too has a left adjoint and is finitarily $2$-monadic. So we have a pair of finitarily $2$-monadic adjunctions as on the left in
\begin{equation*}
\cd{
      [\mathbb N, \Cat] \ar@<4.5pt>[r]^-{G} \ar@{}[r]|-{\bot}   &
      \ar@<4.5pt>[l]^-{V} \Endsf(\Cat) \ar@<4.5pt>[r]^-{H} \ar@{}[r]|-{\bot} &
      \ar@<4.5pt>[l]^-{W} \Mndsf(\Cat) & 
      [\mathbb N, \Cat] \ar@<4.5pt>[r]^-{K} \ar@{}[r]|-{\bot}   &
      \ar@<4.5pt>[l]^-{Z} \Mndsf(\Cat)\rlap{ .}
}
\end{equation*}
Now taking $Z = VW$ and $K = HG$ we obtain a further finitary adjunction as on the right; this is not \emph{a priori} $2$-monadic, but turns out to be so by a direct application of~\cite[Theorem~2]{Lack1999On-the-monadicity}. Consequently, if presented with a strongly finitary $2$-monad on $\Cat$, we may regard it either as an $WH$-algebra or  a $ZK$-algebra; and in either guise may ask whether, as an algebra, it is semiflexible, flexible, or pie. 

The counit of the adjunction $G \dashv V$ yields a morphism of $2$-comonads $KZ = HGVW \to HW$. From  this fact, together with Lemma~\ref{lem:rhoobj}, we deduce the existence of a morphism of $2$-comonads from the $Q$ associated with $ZK$ to that associated with $WH$; so that if a strongly finitary $2$-monad is semiflexible, flexible or pie as a $ZK$-algebra, then it is correspondingly so as a $HW$-algebra. This justifies our calling a $2$-monad \emph{strongly} semiflexible, flexible or pie in the former case, with the corresponding unqualified name serving in the latter one.

We now describe what is meant in saying that a strongly finitary $2$-monad is pie-presentable.  Such 2-monads have been defined only vaguely so far, and by reference to their algebras, which are to be categories equipped with various kinds of structure: operations and transformations, both basic and derived and with equations imposed between derived transformations.  To make precise our notion we will define these latter terms; showing that the strongly finitary 2-monads whose algebras are categories so equipped are exactly those admitting a specific colimit presentation in terms of free monads. The possibility of presenting monads in this way was first discussed in detail in~\cite{Kelly1993Adjunctions}; the $2$-monad case was considered in~\cite[Section 6.4]{Lack2007Homotopy-theoretic} and~\cite[Section 5]{Lack2010A-2-categories}. 

To draw the correspondence between colimits of free 2-monads and their algebras one uses the endomorphism 2-monad $\left<C, C\right> \in \Mnd(\Cat)$ of a category $C$, which has value $[C^{D},D]$ at a category $D$.  The key property is that 2-monad maps $T \to \left<C, C\right>$ correspond with $T$-algebra structures on $C$; it follows that one can understand the algebras for a colimit of 2-monads in terms of the algebras for its constituent parts.  
As the inclusion $\Mndsf(\Cat) \to \Mnd(\Cat)$ has a right adjoint the same holds for colimits of strongly finitary 2-monads, except we now use the strongly finitary coreflection $\left<C, C\right>_{\textnormal{sf}}$ of $\left<C, C\right>$, which in particular has value $[C^{n},C]$ for $n \in \mathbb N$.

Given a discrete signature $\Sigma_1$---that is, a discrete object of $[\mathbb N, \Cat]$---we say that a category $C$ has \emph{basic $\Sigma_{1}$-operations} if it is equipped with a functor $C^n \to C$ for each object of $\Sigma_{1}(n)$.  To so equip $C$ is to give a signature morphism $ \Sigma_{1} \to Z\left<C, C\right>_{\textnormal{sf}}$ and so by adjointness a monad morphism $c \colon K\Sigma_{1} \to \left<C, C\right>_{\textnormal{sf}}$, equally amounting to a $K\Sigma_{1}$-algebra structure on $C$.

In this situation, we have underlying the monad morphism $c$ a signature morphism $Z c \colon ZK\Sigma_1 \to Z\left<C, C\right>_{\textnormal{sf}}$ which equips $C$ with basic $ZK\Sigma_1$-operations. Precomposing with the (monic) unit map $\eta \colon \Sigma_1 \to ZK\Sigma_1$ we re-find amongst these the specified basic $\Sigma_1$-operations on $C$; but also further operations, obtained by substituting and reindexing these basic ones, that are necessarily present because $ZK\Sigma_1$ underlies a $2$-monad. Let us agree to call these \emph{derived $\Sigma_1$-operations}
on $C$: thus for each $t \in K\Sigma_1(n)$ we have the derived $\Sigma_1$-operation $\llbracket t \rrbracket \colon C^n \to C$, the value of $t$ under $c_n \colon K\Sigma_1(n) \to \left<C,C\right>_\mathrm{sf}(n) = [C^n, C]$.

Now suppose we are given another discrete signature $\Sigma_2$ and a pair of signature morphisms $s, t \colon \Sigma_2 \rightrightarrows ZK\Sigma_1$. A category  $C$ with basic $\Sigma_1$-operations will be said to have \emph{basic $\Sigma_2$-transformations} if it comes equipped with a natural transformation $\llbracket s(x) \rrbracket \Rightarrow \llbracket t(x) \rrbracket \colon C^n \to C$ between derived $\Sigma_1$-operations for each $x \in \Sigma_2(n)$. To equip a category $C$ with basic $\Sigma_1$-operations and basic $\Sigma_2$-transformations is to give a $2$-cell of $[\mathbb N, \Cat]$ as on the left in:
\begin{equation*}
\cd[@!C@C-3em@R-1.5em]{
  & ZK\Sigma_1 \ar[dr]^-{ Zc} \dtwocell{dd}{} \\
  \Sigma_2 \ar[ur]^s \ar[dr]_t & & **{!/l 1em/}{Z\left<C, C\right>_\mathrm{sf}} \\
  & ZK\Sigma_1 \ar[ur]_-{ Zc}
} \ \ 
\cd[@!C@C-2em@R-1.5em]{
  & K\Sigma_1 \ar[dr]^-{c} \dtwocell{dd}{} \\
  K\Sigma_2 \ar[ur]^{\overline s} \ar[dr]_{\overline t} & & \left<C, C\right>_\mathrm{sf} \\
  & K\Sigma_1 \ar[ur]_-{c}
} \ \ 
\cd[@!C@C-1.2em@R-1.5em]{
  & K\Sigma_1 \ar[dr]^{q} \dtwocell{dd}{\gamma} \\
  K\Sigma_2 \ar[ur]^{\overline s} \ar[dr]_{\overline t} & & R\rlap{ ;} \\
  & K\Sigma_1 \ar[ur]_{q}
}
\end{equation*}
or equally, by adjointness, a $2$-cell in $\Mndsf(\Cat)$ as in the middle; or equally, a morphism $d \colon R \to \left<C, C\right>_\mathrm{sf}$ in $\Mndsf(\Cat)$ out of the coinserter  of $\overline s$ and $\overline t$ as on the right; or equally, an   $R$-algebra structure on $C$.

Now the coinserter map $q \colon K\Sigma_1 \to R$, like any coinserter map, is objective; and we will see below that the 2-monad $ZK$ preserves objectives, so that by Proposition~\ref{prop:whenpresbij}, $Zq$ is also objective, which is to say, pointwise bijective on objects. Thus objects of $R(n)$ coincide with objects of $K\Sigma_1(n)$, that is, with derived $\Sigma_1$-operations, a fact which would no longer be true in the world of finitary 2-monads on \Cat. On the other hand, $R(n)$ will almost certainly have different \emph{morphisms} to $K\Sigma_1(n)$, and these will be the concern of our next  definition.

For a category $C$ with basic $\Sigma_{1}$-operations and $\Sigma_2$-transformations, amounting to an $R$-algebra structure as above, we define a \emph{derived $\Sigma_2$-transformation} $\llbracket \alpha \rrbracket \colon \llbracket s \rrbracket \Rightarrow \llbracket t \rrbracket \colon C^n \to C$ between derived $\Sigma_1$-operations to be the image under the monad map $d \colon R \to \left<C, C\right>_\mathrm{sf}$ of some morphism $\alpha \colon s \to t$ of $R(n)$. 
By precomposing with the adjoint transpose of the coinserter $2$-cell $\gamma$, we re-find amongst such derived transformations the basic $\Sigma_2$-transformations with which $C$ was equipped; but also others, obtained by substituting and composing together the basic ones, that are necessarily present because $R$ is a $2$-monad. 

Finally, let us suppose given a third discrete signature $\Sigma_3$ and a parallel pair of $2$-cells $\alpha$, $\beta$ as on the left of:
\begin{equation*}
\cd[@C+1em]{
  \Sigma_3 \ar@/^12pt/[r]^{h} \ar@/_12pt/[r]_{k} \dtwocell[0.32]{r}{\alpha} \dtwocell[0.6]{r}{\beta} &
  ZR
} \qquad
\cd[@C+1em]{
  K\Sigma_3 \ar@/^12pt/[r]^{\overline h} \ar@/_12pt/[r]_{\overline k} \dtwocell[0.32]{r}{\overline \alpha} \dtwocell[0.6]{r}{\overline \beta} &
  R
}
 \qquad
\cd[@C+1em]{
  K\Sigma_3 \ar@/^12pt/[r]^{\overline h} \ar@/_12pt/[r]_{\overline k} \dtwocell[0.32]{r}{\overline \alpha} \dtwocell[0.6]{r}{\overline \beta} &
  R \ar[r]^-{r} & T\rlap{ .}
}
\end{equation*}
A category $C$ with basic $\Sigma_1$-operations and $\Sigma_2$-transformations is said to \emph{satisfy $\Sigma_3$-equations} if $\llbracket \alpha_x \rrbracket = \llbracket \beta_x \rrbracket\colon \llbracket h(x) \rrbracket \Rightarrow \llbracket k(x) \rrbracket$ for each $x \in \Sigma_3(n)$. Forming the adjoint transposes in $\Mndsf(\Cat)$ of $\alpha$ and $\beta$, as in the centre above, this is equally well to ask that the monad map $d \colon R \to \left<C, C\right>_\mathrm{sf}$ encoding $C$'s $R$-algebra
structure should satisfy $d \circ \overline \alpha = d \circ \overline \beta$; or equally, that $C$ should bear algebra structure for $T$, the coequifier in $\Mndsf(\Cat)$ of $\overline \alpha$ and $\overline \beta$ as on the right above.

Thus basic operations on a category $C$ are encoded by its being an algebra for the free monad $K\Sigma_1$ on a discrete signature $\Sigma_{1}$, basic transformations between derived operations by its being an algebra for the coinserter $R$ of a pair of maps $K\Sigma_{2} \rightrightarrows K\Sigma_{1}$ with $\Sigma_{2}$ discrete, and equations between derived transformations by its being an algebra for the coequifier $T$ of a pair of 2-cells between maps $K\Sigma_{3} \rightrightarrows R$ with $\Sigma_{3}$ also discrete.  We define a strongly finitary 2-monad $T$ to be \emph{pie-presentable} when it admits such a colimit presentation.

With this in place, we are in a position to state our main result:
\begin{Theorem}\label{thm:piepres}
A strongly finitary $2$-monad on $\Cat$ is pie-presentable if and only if it is strongly pie.
\end{Theorem}
One direction is easy: for by construction, any pie-presentable strongly finitary $2$-monad lies in the closure of the free $ZK$-algebras under pie colimits, and is therefore strongly pie by Proposition~\ref{thm:classesofpie}. For the other direction, we will apply the result of the previous section;  in preparation for which, we will now analyse the action of the $2$-monad $ZK$ further.

Given a signature $\Sigma \in [\mathbb N, \Cat]$, the free $2$-monad $K\Sigma$ on it is equally well the free $2$-monad on the endofunctor $G\Sigma = \sum_{n \in \mathbb N} \Sigma(n) \times (\thg)^n$. To describe this, consider $\textnormal{$\Sigma$-Alg}$, the $2$-category of algebras for the endofunctor $G\Sigma$; its objects are categories $C$ equipped with functors $\Sigma(n) \times C^n \to C$ for each $n \in \mathbb N$, and its $1$- and $2$-cells are the evident structure-preserving maps. There is a forgetful $2$-functor $\textnormal{$\Sigma$-Alg} \to \Cat$; this has a left $2$-adjoint, and by the argument of~\cite[Proposition 22.2]{Kelly1980A-unified}, the induced $2$-monad on $\Cat$ is precisely $K\Sigma$.

We now describe the left $2$-adjoint of $\textnormal{$\Sigma$-Alg} \to \Cat$, thus the free $\Sigma$-algebra on a category $C$. The description is essentially standard, and can be found at various levels of generality in~\cite[Section~2.2]{Kock2011Polynomial}, \cite[Appendix~D]{Leinster2004Higher} or~\cite[Theorem~24]{Gambino2004Wellfounded}, for example. First, let $\Omega$ be the set inductively defined by the following  clauses:
\begin{itemize}
\item $\star \in \Omega$;
\item Whenever $n \in \mathbb N$ and $\alpha_1, \dots, \alpha_n \in \Omega$, then also $(\alpha_1, \dots, \alpha_n) \in \Omega$.
\end{itemize}
Next, we recursively associate to each element $\alpha \in \Omega$ a natural number $\abs \alpha$ and an object $\hat \alpha \in [\mathbb N, \Cat]$, as follows:
\begin{align*}
    \abs{\alpha} & = \begin{cases} 1 & \text{if $\alpha = \star$;} \\
    \abs{\alpha_1} + \dots + \abs{\alpha_n} & \text{if $\alpha = (\alpha_1, \dots, \alpha_n)$;}
    \end{cases}\\
\text{and} \qquad
    \hat{\alpha} & = \begin{cases} 0 & \text{if $\alpha = \star$;} \\
    \hat{\alpha}_1 + \dots + \hat{\alpha}_n + y_n & \text{if $\alpha = (\alpha_1, \dots, \alpha_n)$.}
\end{cases}
 \end{align*}
Here, $y_n \in [\mathbb N, \Cat]$ is the representable at $n$. Now the free $\Sigma$-algebra on $C$ will have underlying category
\begin{equation}\label{eq:sum}
    C^\ast = \sum_{\alpha \in \Omega} [\mathbb N, \Cat](\hat \alpha, \Sigma) \times {C}^{\abs \alpha}\rlap{ ;}
\end{equation}
to give its $\Sigma$-algebra structure, we must give functors $\Sigma n \times (C^\ast)^n \to C^\ast$ for each $n$. Unfolding the definition~\eqref{eq:sum}, this is equally to give a functor
\begin{equation*}
    \Sigma n \times [\mathbb N, \Cat](\hat{\alpha_1}, \Sigma) \times \dots \times [\mathbb N, \Cat](\hat{\alpha_n}, \Sigma) \times C^{\abs{\alpha_1}} \times \dots \times C^{\abs{\alpha_n}} \to C^\ast
\end{equation*}
for every $n \in \mathbb N$ and $\alpha_1, \dots, \alpha_n \in \Omega$. On defining $\alpha = (\alpha_1, \dots, \alpha_n) \in \Omega$, we observe that the domain of this functor is isomorphic to $[\mathbb N, \Cat](\hat \alpha, \Sigma) \times C^{\abs{\alpha}}$, so that we may take it to be the coproduct injection at $\alpha$.
This makes $C^\ast$ into a $\Sigma$-algebra, which may now be shown to be the free $\Sigma$-algebra on $C$. It follows that~\eqref{eq:sum} gives the value at $C$ of the free $2$-monad $K\Sigma$ on $\Sigma$, and we conclude that:
\begin{Proposition}\label{prop:explicitdesc}
The $2$-functor $ZK$ on $[\mathbb N, \Cat]$ is given by:
\begin{equation}\label{eq:zkform}
    (ZK\Sigma)(n) = \sum_{\alpha \in \Omega} [\mathbb N, \Cat](\hat \alpha, \Sigma) \times n^{\abs{\alpha}}\rlap{ .}
\end{equation}
\end{Proposition}

This proposition exhibits $ZK$ as pointwise a coproduct of hom-functors of the form $[\mathbb N, \Cat](\hat \alpha, \thg)$. Since each $\hat \alpha$ is projectively discrete in $[\mathbb N, \Cat]$, every such hom-functor will preserve objective morphisms, whence so also will $ZK$ itself. In fact, more is true: every such hom-functor, and hence also $ZK$, will preserve maps which are pointwise bijective on objects \emph{and full}. We will use this fact shortly.

Finally, we observe that the monad $(ZK)_d$ induced by $ZK$ on $[\mathbb N, \Cat]_d = [\mathbb N, \Set]$ sends an object $\Sigma$ to $\sum_{\alpha \in \Omega} [\mathbb N, \Set](\hat \alpha, \Sigma) \times n^{\abs{\alpha}}$; it is thus also pointwise a coproduct of representable functors, and so preserves connected limits, in particular coreflexive equalisers. 
So all of the hypotheses of Theorem~\ref{thm:classescoincide} are satisfied, which allows us to conclude that every pie $ZK$-algebra---that is, every strongly pie strongly finitary $2$-monad---is a pie colimit of frees on discrete signatures. However, this is not quite enough to show that every such $2$-monad is pie-presentable, as we must to complete the proof of Theorem~\ref{thm:piepres}; for that, we need to produce a colimit presentation of the specific form demanded in the definition of pie-presentability. We will do so by adapting the proof of Proposition~\ref{thm:1in2}.
\begin{proof}[Proof of Theorem~\ref{thm:piepres}]
We are to show that a strongly pie strongly finitary $2$-monad $T$ is pie-presentable. By  Theorem~\ref{thm:classescoincide}, we know that such a $T$ admits an objective morphism $f \colon K\Sigma_1 \to  T$ in $\Mndsf(\Cat)$, where $\Sigma_1$ is a discrete signature.  We now trace through the proof of Proposition~\ref{thm:1in2} to derive from this a pie-presentation of $T$. Arguing as in the first part of that proof, we form the comma object
in $[\mathbb N, \Cat]$ as on the left in:
\begin{equation*}
\cd{
Y \ar[r]^d \ar[d]_c \dtwocell{dr}{\gamma} & ZK\Sigma_1 \ar[d]^{Zf} \\
ZK\Sigma_1 \ar[r]_{Zf} & ZT
} \qquad  \qquad
\cd{
K\Sigma_2 \ar[r]^{f} \ar[d]_{g} \dtwocell{dr}{\delta} & K\Sigma_1 \ar[d]^{q} \\
K\Sigma_1 \ar[r]_{q} & R\rlap{ .}
}\end{equation*}
We set $\Sigma_2 = DOY$, the discrete coreflection of $Y$, and take $f, g \colon K\Sigma_2 \rightrightarrows K\Sigma_1$ to be the transposes of $d \circ \lambda_Y$ and $c \circ \lambda_Y$. Now on forming the coinserter of $f$ and $g$ as on the right above, the argument of Proposition~\ref{thm:1in2} yields an induced map $r \colon R \to T$ in $\Mndsf(\Cat)$ for which $Zr$ is pointwise bijective on objects and full. But now $ZKZr$ is also pointwise bijective on objects and full, since as we observed above, $ZK$ preserves such morphisms. So to complete the construction, we argue as in the final part of the proof of Proposition~\ref{thm:1in2}. We form the equi-kernel of $Zr$, as on the left of:
\begin{equation*}
\cd[@C+1em]{
  W \ar@/^12pt/[r]^{w} \ar@/_12pt/[r]_{z} \dtwocell[0.32]{r}{\phi} \dtwocell[0.6]{r}{\psi} &
  ZR \ar[r]^{Zr} & ZT
} \qquad \qquad
\cd[@C+1em]{
  K\Sigma_3 \ar@/^12pt/[r]^{h} \ar@/_12pt/[r]_{k} \dtwocell[0.35]{r}{\alpha} \dtwocell[0.62]{r}{\beta} &
  R \ar[r]^{r} & T\rlap{ .}
}
\end{equation*}
Now we set $\Sigma_3 = DOW$, and on whiskering with $\lambda_W \colon DOW \to W$ and taking transposes, obtain data $h$, $k$, $\alpha$, $\beta$ as on the right; now the argument of the last part of Proposition~\ref{thm:1in2} shows that $r$ exhibits $T$ as the coequifier of $\alpha$ and $\beta$, so that $T$ is pie-presentable as required.
 \end{proof}
Let us now relate the notion of strongly pie $2$-monad with our intuition that the pie algebras are those which are ``free at the level of objects''. If $T$ is a strongly finitary $2$-monad on $\Cat$, then we have, as before, an induced monad $T_d$ on $\Cat_d = \Set$; and it is easy to see that this $T_d$ is finitary. Let us say that a finitary monad on $\Set$ is \emph{free on a signature} if it is in the essential image of the left adjoint of the (monadic) forgetful functor $\Mndf(\Set) \to [\mathbb N, \Set]$.

\begin{Proposition}\label{prop:stronglypie}
A strongly finitary $2$-monad $T$ on $\Cat$ is strongly pie just when the induced finitary monad $T_d$ on $\Set$ is free on a signature.
\end{Proposition}
For an earlier result in a similar spirit to this one, see \cite[Proposition 4.3]{Kelly1974Coherence}.
\begin{proof}
An analysis identical in form to the one leading to Proposition~\ref{prop:explicitdesc} yields a description of the free finitary monad on $\Set$ generated by a signature $\Sigma \in [\mathbb N, \Set]$, and from this we deduce that there is a commuting diagram of left and right adjoints as on the left in:
\begin{equation*}
\cd[@!C@R+0.5em@C-2em]{
 \Mndsf(\Cat)_{0} \ar[rr]^{(-)_{d}} \ar@<-3pt>[dr]_-{OZ}  & &
 \Mndf(\Set)\ar@<3pt>[dl]^-{} \\ &
[\mathbb N, \Set] \ar@<-3pt>[ul]_-{KD} \ar@<3pt>[ur]^-{}}
\qquad      \cd{
        \Mndsf(\Cat)_0 \ar[r]^{(\thg)_d} \ar[d]_{} & \Mndf(\Set) \ar[d]^{} \\
        \ZKoAlg \ar[r]_j  & \ZKpAlg\rlap{ .}
    }
\end{equation*}
 Consider now the  square on the right above, in which the unlabelled vertical arrows are the canonical equivalences. Each vertex admits an adjunction with $[\mathbb N, \Set]$, whilst each morphism commutes with both the left and the right adjoint parts of these adjunctions. As there is a unique functor to a category of algebras commuting with both adjoints it follows that the square commutes.  Thus to say of $T \in \Mndsf(\Cat)$ that $T_d$ is free on a signature is equally well to say that $jAT$ is a free $\Tp$-algebra; and so equally, by Theorem~\ref{thm:charthm}, that $T$ is a strongly pie $2$-monad.
\end{proof}
For our final result, we fulfil a promise made at the end of Section~\ref{sec:catexs} by giving a general characterisation of the pie algebras for any strongly pie $2$-monad.
\begin{Proposition}\label{prop:stronglypiepiealgs}
If $T$ is a strongly pie strongly finitary $2$-monad on $\Cat$, then the following classes of $T$-algebras coincide:
\begin{enumerate}
\item The closure of the frees (on discretes) under pie colimits;
\item The objective quotients of the frees (on discretes);
\item The pie algebras;
\end{enumerate}
and may be characterised as the $T$-algebras $A$ for which $jA$ is a free $\Tp$-algebra.
\end{Proposition}
\begin{proof}
To say that $T \colon \Cat \to \Cat$ is strongly finitary is to say that it is the left Kan extension along the inclusion $\mathbb F \hookrightarrow \Cat$ of some $D \colon \mathbb F \to \Cat$. Thus $T \cong \int^{n \in \mathbb F} Dn \times \Cat(n,\thg)$ is a colimit of functors of the form $\Cat(n,\thg)$, and since every such functor preserves objectives, so too does $T$.
Moreover, since $T$ is strongly pie, the induced monad $T_d$ on $\Set$ is free on a signature. Any such monad is a coproduct of representables, and as such preserves all connected limits, and in particular coreflexive equalisers. Thus all the hypotheses of Theorem~\ref{thm:classescoincide} are satisfied and the result follows.
\end{proof}

\bibliographystyle{acm}

\bibliography{bibdata}

\end{document}